\documentclass[11pt]{article}

\usepackage{mathtools, amsthm, accents, tikz, amssymb, latexsym, amsmath, amscd, amsfonts, array, lmodern, enumerate, stmaryrd, rotating, caption, graphicx, hyperref, float,tikz-cd, color, yfonts}
\usepackage[all]{xy}
\CompileMatrices
\hypersetup{nesting=true,debug=true,naturalnames=true}

\def\y{\,\vert\,}
\def\etc{\cdots\hspace{-.3mm}}

\def\F{\protect\operatorname{Conf}}
\def\DF{\protect\operatorname{DConf}}

\newtheorem{proposition}{Proposition}[section]
\newtheorem{corollary}[proposition]{Corollary}
\newtheorem{definition}[proposition]{Definition}
\newtheorem{theorem}[proposition]{Theorem}
\newtheorem{remark}[proposition]{Remark}
\newtheorem{example}[proposition]{Example}
\newtheorem{lemma}[proposition]{Lemma}

\begin{document}

\title{An algorithmic discrete gradient field for non-colliding cell-like objects and the topology of pairs of points on skeleta of simplexes}

\author{Emilio J.~Gonz\'alez and Jes\'us Gonz\'alez}

\date{}

\maketitle

\begin{abstract}
For a positive integer $n$ and a finite simplicial complex $K$, we describe an algorithmic procedure constructing a maximal discrete gradient field $W(K,n)$ on Abrams' discretized configuration space $\DF(K,n)$. Computer experimentation shows that the field is generically optimal. We study the field $W(K,n)$ for $n=2$ and $K=\Delta^{m,d}$, the $d$-dimensional skeleton of the $m$-dimensional simplex. In particular, we prove that $\DF(\Delta^{m,d},2)$ is $(\min\{d,m-1\}-1)$-connected, has torsion-free homology and admits a minimal cell structure. We compute the Betti numbers of $\DF(\Delta^{m,d},2)$ and, for certain values of $d$, we prove that $\DF(\Delta^{m,d},2)$ breaks, up to homotopy, as a wedge of (not necessarily equidimensional) spheres.
\end{abstract}



\section{Introduction and main results}\label{mainresult}

For a cell complex $X$ and a positive integer $n$, a discrete analogue of the configuration space
\begin{equation}\label{configurationspace}
\F(X,n)=\{(x_1,\ldots,x_n)\in X^n\colon x_i\neq x_j \text{ for } i\neq j\}
\end{equation}
was introduced in Abrams Ph.D.~thesis~\cite{MR2701024}. Abrams' model, denoted here by $\DF(X,n)$, is the subcomplex of the product complex $X^n$ resulting by removing all open cells whose closure intersects the fat diagonal $X^n\setminus\F(X,n)$. In other words, $\DF(X,n)$ is the largest subcomplex of~$X^n$ contained in $\F(X,n)$.

Abrams discrete configuration spaces are particularly interesting objects on their own right, that have proven to be a highly valuable tool for understanding the algebraic topology properties of classifying spaces of graph braid groups, i.e., of spaces~(\ref{configurationspace}) for 1-dimensional complexes $X$ (\cite{MR2949126,MR2833585,MR3426912,MR3000570}). In addition, and from a more practical viewpoint, discrete configuration spaces have certain advantages over their classical counterparts. For instance, consider $n$~autonomous robots moving on a system of tracks forming a graph $G$. Then, by using $\DF(G,n)$ as a model for the corresponding collision-free motion planning problem, we force a \emph{safety} rule, namely, the requirement that the shortest path among any two robots always includes a complete edge of $G$. This allows us to replace the (dimensionless, in principle) notion of ``robot'' by actual trains whose lengths are no larger than the shortest edge in~$G$.

In general, the homotopy type of $\DF(X,n)$ may be different from that of $\F(X,n)$. When $\dim(X)=1$, there are well understood (subdivision-type) conditions assuring that Abrams discrete model captures the homotopy type of its classical counterpart (\ref{configurationspace}), see~\cite[Theorem~2.4]{MR2833585} and \cite{MR3276733}. However, for larger dimensional complexes $X$, no analogous conditions are known implying a potential homotopy equivalence $\F(X,n)\simeq\DF(X,n)$. The situation for $n=2$ and~$X$ an arbitrary \emph{simplicial} complex is exceptional, as it has been known for a long time that $\DF(X,2)$ sits inside the ``deleted product'' $\F(X,2)$ as a strong deformation retract~(see~\cite[Theorem~11.2]{MR113226} for the case of a finite $X$ and, for a general simplicial complex,~\cite[Lemma 2.1]{MR89410}, noticing the fix in~\cite{aronestrickland}). This is particularly appealing, as deleted products are important objects playing a central role, for instance, in the Euclidean embedding problem for complexes~(\cite{MR89410}).

As suggested in~\cite{MR3003699}, a systematic study of discrete configuration spaces could then start by considering the case of simplices, the basic building blocks of complexes. In such a direction, Abrams, Gay and Hower have shown that the discrete configuration space $\DF(\Delta^m,n)$ associated to an $m$-dimensional simplex $\Delta^m$ splits, up to homotopy, as a wedge sum of spheres all of the same dimension. While their techniques are based on classical homotopy theory, the problem of finding a combinatorially flavored argument arises (see~\cite[page~1085]{MR3003699}). In this paper we pave the way for addressing such a challenge (see Theorem~\ref{maintheoremmaximal} below). In particular, for $1\leq d\leq m$, we study, from a purely computational topology perspective, the homotopy type of $\DF(\Delta^{m,d},2)$, where $\Delta^{m,d}$ stands for the $d$-dimensional skeleton of $\Delta^m$. This leads us to a major generalization of the results in~\cite{algorithm1} regarding the algebraic topology of 2-point configuration spaces on complete graphs (i.e., 1-dimensional skeleta $\Delta^{m,1}$).

Much of our interest in the complex $\Delta^{m,d}$ stems from the fact that any finite simplicial complex~$X$ can be thought of as a subcomplex of $\Delta^{m,d}$, where $d=\dim(X)$ and $m$ is one less than the number of vertices of $X$. Thus, Abrams' $\DF(X,2)$ (or $\DF(X,n)$, for that matter) becomes a subcomplex of $\DF(\Delta^{m,d},2)$ (respectively, of $\DF(\Delta^{m,d},n)$). In other words, instead of considering basic building blocks of complexes, i.e., the Abrams-Gay-Hower motivation, we propose to focus on universally large complexes.

Before stating in detail our main contributions, we briefly summarize them. We prove that $\DF(\Delta^{m,d},2)$ has torsion free homology admiting a minimal cell structure in the sense of~\cite[Section 4.C]{MR1867354} ---even in the non-simply-connected case. We compute the Betti numbers of $\DF(\Delta^{m,d},2)$ and show that, for $d\geq m-3$, $\DF(\Delta^{m,d},2)$ is homotopy equivalent to a wedge of spheres. Our calculations suggest that a homotopy decomposition of $\DF(\Delta^{m,d},2)$ as a wedge of spheres holds as long as
\begin{equation}\label{borde}
d>\max\left\{1,\frac{m-2}{2}\right\}.
\end{equation}
Although we do not prove such assertion, we provide indirect evidence in Remark~\ref{evidencia} below. 

\begin{remark}\label{abramsobs}{\em
A homotopy decomposition of $\DF(\Delta^{m,d},2)$ as a wedge of spheres does not hold outside the range in~(\ref{borde}). For instance, as proved in~\cite{algorithm1}, nontrivial cup products exist in the cohomology ring of 2-point configuration spaces of complete graphs (i.e., the case $d=1$ above). Likewise, there are nontrivial cup products in $\DF(\Delta^{m,d},2)$ for $d=\frac{m-2}{2}$, though the situation is much subtler. Indeed, for $d\geq1$,
\begin{equation}\label{familiad}
\DF(\Delta^{2d+2,d},2)
\end{equation}
turns out to be homeomorphic to an orientable closed manifold, so nontrivial cup products hold as part of Poincar\'e duality phenomena. Manifold properties of~(\ref{familiad}) will be the topic of a separate paper, and here we only remark a couple of appealing facts. For starters, Abrams' homeomorphism $\DF(K_5,2)\cong S_6$ (\cite[Example~2.4]{MR2701024}) is just the first instance in the family of manifolds~(\ref{familiad}), where $K_5=\Delta^{4,1}\!$, the complete graph on five vertices, and $S_6$ stands for an orientable closed surface of genus six. In addition, while $S_6$ is a connected sum of six torii, item \emph{(\ref{it4})} in the next theorem shows\footnote{We thank Daciberg L. Gon\c calves for bringing such a fact to our attention.} that (\ref{familiad}) has the homology type of a  connected sum of $\binom{2d+2}{d+1}$ copies of $S^d\times S^d$.
}\end{remark}

Our main theoretical result is spelled out next.

\begin{theorem}\label{maintheoremlarge}
For $1\leq d\leq m\geq2$, $\DF(\Delta^{m,d},2)$ is
\begin{itemize}
\item a path-connected space with torsion free homology, which is
\item homotopy equivalent to a  cell complex having as many cells in a given dimension~$i$ as the $i$-th Betti number $\beta_i(m,d)$ of \,$\DF(\Delta^{m,d},2)$.
\end{itemize}
Specifically, setting $s_{m,d}:=\min\{d,m-1\}$ and $\beta_{m,d}:=\beta_{s_{m,d}}(m,d)$, we have:
\begin{enumerate}[(I)]
\item\label{it1} $\DF(\Delta^{m,d},2)$ is $(s_{m,d}-1)$-connected.
\item\label{it2} For $d\in\{m-1,m\}$, $$\DF(\Delta^{m,d},2)\simeq S^{m-1}.$$
\item\label{it3} For $d=m-2$,
\begin{equation}\label{dm2}
\DF(\Delta^{m,m-2},2)\simeq \bigvee_{\beta_{m,m-2}}S^{m-2},
\end{equation}
where $\beta_{m,m-2}=2m+1$.
\item\label{it4} For $\frac{m-2}2\leq d\leq m-3$ (so that $m\geq4$),
\begin{equation}\label{untop}
\DF(\Delta^{m,d},2)\simeq \left(\hspace{.5mm}\bigvee_{\beta_{m,d}}S^{d}\right)\cup_{\alpha_{m,d}} e^{m-2},
\end{equation}
where $\beta_{m,d}=2\binom{m}{\,d+1\,}$. Furthermore, the attaching map $\alpha_{m,d}\colon S^{m-3}\to\bigvee_{\beta_{m,d}}S^{d}$ is null-homotopic provided $1<d=m-3$, so that
\begin{equation}\label{dm3}
\DF(\Delta^{m,m-3},2)\simeq \left(\hspace{.5mm}\bigvee_{m(m-1)}S^{m-3}\right)\vee S^{m-2}
\end{equation}
for $m\geq5$.

\item\label{it5} For $d<\frac{m-2}2$,
\begin{equation}\label{mutop}
\DF(\Delta^{m,d},2)\simeq \left(\bigvee_{\beta_{m,d}}S^{d}\right)\cup\left(\bigcup_{\beta_{2d}(m,d)}e^{2d}\right),
\end{equation}
where $\beta(m,d)=2\binom{m}{\,d+1\,}$ and
$$\beta_{2d}(m,d)=\binom{m}{\,d+1\,}\binom{\,m-d-1\,}{d+1}-\hspace{-.3mm}\sum\limits_{b=m-d\;}^{m}\;\sum\limits_{a=b-d-1}^{b-1}\hspace{-.5mm}\binom{a-1}{\,b+d-m\,}\binom{\,a-b+m-d-1\,}{d+1-b+a}.$$
\end{enumerate}
\end{theorem}

The easy and well-known situation for $d\in\{m,m-1\}$ in item \emph{(\ref{it2})} is included for completeness and comparison purposes. Note that~(\ref{dm2}) is formally identical to~(\ref{dm3}), as in both cases the total number of spheres in the wedge sum is $2\binom{m}{\hspace{.3mm}d+1\hspace{.2mm}}+1$, except that all spheres happen to be of the same dimension when $d=m-2$. Note also that, for $d=m-2=1$, item \emph{(\ref{it3})} recovers the fact in~\cite[Corollary~4.2]{algorithm1} that the configuration space of two ordered points in $\Delta^{3,1}=K_4$, the complete graph on four vertices, is homotopy equivalent to a wedge of seven circles.

\begin{remark}\label{evidencia}{\em
We have already noted (Remark~\ref{abramsobs}) that the attaching map $\alpha_{m,d}$ in~\emph{(\ref{it4})} is essential (and rather interesting) for $d=\frac{m-2}{2}$, but null-homotopic for $1<d=m-3$. Additionally, in the range 
\begin{equation}\label{rangointeresante}
\frac{m-2}{2}<d<m-3,
\end{equation}
$\alpha_{m,d}$ is a stable map, so it is forced to be null-homotopic whenever $m-d-3\in\{4,5,12,61\}$, i.e., when the relevant stable stem is known to vanish (\cite{MR0143217,MR3702672}). It would be interesting to know if there are further values of $d$ in the range~(\ref{rangointeresante}) for which $\alpha_{m,d}$ is null-homotopic, so that~(\ref{untop}) becomes a wedge sum of spheres. On the other hand, we believe that nonzero cup products will always exist in the case of~(\ref{mutop}) ---checking such a situation is within reach with the tools in this paper---, so that some component of the attaching map in~(\ref{mutop}) would not be null-homotopic.
}\end{remark}

\begin{remark}\label{superficie}{\em
The main result in~\cite{MR3003699}, namely, the fact that $\DF(\Delta^m,n)$ is homotopy equivalent to a wedge of spheres of dimension $m-n+1$, suggests that (parts of) Theorem~\ref{maintheoremlarge} would admit suitable extensions for $\DF(\Delta^{m,d},n)$ when $n\geq2$. 
}\end{remark}


We prove Theorem~\ref{maintheoremlarge} by using Forman's discrete Morse theory techniques. In fact, as a second main contribution in this work, we describe an algorithmic construction of an efficient discrete gradient field $W(X,n)$ on $\DF(X,n)$, for any finite simplicial complex $X$ and any $n\geq2$.

\begin{theorem}\label{maintheoremmaximal}
The gradient field $W(X,n)$ is maximal in the sense that no $(p-1)$-dimensional face and no $(p+1)$-dimensional coface of a critical $p$-dimensional face is critical.
\end{theorem}

The gradient field $W(X,n)$ is inspired by the construction in~\cite{algorithm1} of a maximal gradient field $W(X)$ on any finite simplicial complex $X$. As described in~\cite{algorithm1}, $W(X)$ can be constructed through two different algorithms $\mathcal{A}$ and $\overline{\mathcal{A}}$, where the former one is suitable for studying theoretical properties of $W(X)$, while the latter one leads to a fast computer implementation. Our construction of $W(X,n)$ is a blend of the two algorithms $\mathcal{A}$ and $\overline{\mathcal{A}}$. Indeed, while $W(X,1)=W(X)$ is the output of $\overline{\mathcal{A}}$, the construction of $W(X,n)$ for $n>1$ follows a strategy similar to the one supporting~$\mathcal{A}$.

The construction of $W(X,n)$ depends on the lexicographic order of faces of $X$ induced by a given total ordering of the vertices of $X$. Yet, the actual constructing process and, therefore, the resulting gradient field, are far more elaborate than the lexicographic gradient field introduced in \cite{MR2770552} to study random Vietoris-Rips complexes above the thermodynamic limit ---later extended to the filtered realm in \cite{MR4298669}. Indeed, in the notation of Section \ref{subsectiondmthry}, all Morse pairings $\alpha\nearrow\beta$ in Kahle's construction have the form $\beta=\alpha\cup\{v\}$ where $v$ is a vertex smaller than any of the vertices of $\alpha$. The fact that such a situation seldom holds in our construction leads to the efficiency property of $W(X,n)$ formalized in Theorem \ref{maintheoremmaximal} and exemplified in Remark \ref{raaghrist} below as well as by the optimality of $W(\Delta^{m,d},2)$.

Most of the assertions in Theorem~\ref{maintheoremlarge} depend on the fact that $W(\Delta^{m,d},2)$ is optimal. The~full strength of Theorem~\ref{maintheoremlarge} then comes from the achievement of two major goals:
\begin{itemize}
\item[(a)] We give an explicit classification of the resulting structure of $\DF(\Delta^{m,d},2)$ into collapsible, redundant and critical cells.
\item[(b)] We describe the dynamics of the resulting gradient paths. 
\end{itemize}

The first task generalizes (to the higher dimensional complexes~$\Delta^{m,d}$) Farley-Sabalka's classification of critical, redundant and collapsible cells in Abrams' discrete homotopy model for graph configuration spaces (see~\cite[Theorem~3.6]{MR2171804} for the unordered case, and~\cite[Theorem~3.2]{MR4356250} for the ordered case). Likewise, the second task generalizes the corresponding goal that lead in~\cite{tere} to a complete understanding of the cohomology ring of full braid groups on trees. Explicit proof details of both tasks are necessarily technical, and explicit computer outputs played a central role in assembling our general arguments.

As a way of illustration, we highlight:

\begin{theorem}[See Corollary~\ref{celdascriticas}]\label{maintheoremclass}
For integers $p$ and $q$, let $\langle p,q\rangle$ stand for the set of integers $\{p,p+1,\ldots,q-1,q\}$. Cells in $\DF(\Delta^{m,d},2)$ are in one-to-one correspondence with ordered pairs $(A,B)$, where $A$ and $B$ are disjoint subsets of $\{0,1,\ldots,m\}$, each of cardinality at most $d+1$, i.e., $|A|\leq d+1\geq|B|$. For $1\leq d\leq m-2$, the cell corresponding to $(A,B)$ is critical if and only if one of the following seven mutually exclusive conditions hold:

\medskip
In dimension 0:
\begin{enumerate}
\item $(A,B)=(\{m-1\},\{m\})$.
\end{enumerate}

In dimension $d$:
\begin{enumerate}\addtocounter{enumi}{1}
\item $B=\{b\}$ with $b\geq1$, $b-1\not\in A$, $\langle b+1,m\rangle\subseteq A$ and $|A|=d+1$.
\item $A=\{m\}$, $\max B<m-1$ and $|B|=d+1$.
\item $\max B=m-1$, $\langle0,m-1\rangle\neq B$, $A=\{\max(\langle0,m-1\rangle-B)\}$, $|B|=d+1$.
\end{enumerate}

In dimension $2d$:
\begin{enumerate}\addtocounter{enumi}{4}
\item $\max A<m$, $\langle1+\max B,m-1\rangle\not\subseteq A$ and $|A|=|B|=d+1$ and $\,2(d+1)\leq m$.
\item $\max(A\cup B)<m$, $\langle1+\max B,m-1\rangle\subseteq A$, $\langle0,\max B\rangle\neq B$, $\max(\langle0,\max B\rangle-B)\not\in A$, $|A|=|B|=d+1$ and $\,2(d+1)<m$.
\end{enumerate}

In dimension $m-2$:
\begin{enumerate}\addtocounter{enumi}{6}
\item $(A,B)=(\langle d+1,m-1\rangle,\langle0,d\rangle)$ and $\,m\leq2(d+1)$.
\end{enumerate}
\end{theorem}

Here is a brief roadmap for the paper. The gradient field $W(X,n)$ is constructed in Section~\ref{optgrafie}, where its maximality property (Theorem~\ref{maintheoremmaximal}) is proved. The characterization of collapsible (Theorems~\ref{teorema1} and~\ref{teorema2}), redundant (Corollary~\ref{celdasredundantes}) and critical cells (Corollary~\ref{celdascriticas}), accomplished in Section~\ref{crigradflo}, leads to a proof of most assertions in Theorem~\ref{maintheoremlarge}. The second assertion in item~\emph{(\ref{it4})} of Theorem~\ref{maintheoremlarge}, i.e., the fact that the attaching map for the top cell in the critical CW model for $\DF(\Delta^{m,m-3},2)$ is homotopically trivial for $m\geq5$, requires the close analysis in Section~\ref{homtopviamorse} of the corresponding gradient flow. In all cases, the optimality of $W(\Delta^{m,d},2)$ and the minimality of the corresponding critical CW structure follow from the triviality of the corresponding Morse differential.

\begin{remark}\label{raaghrist}{\em
The $n$-th symmetric group $\Sigma_n$ acts cellularly on $\DF(X,n)$ by permuting coordinates. The fact that our gradient field $W(X,n)$ is \emph{not} $\Sigma_n$-equivariant should be interpreted as one of its key features. For instance, for $X=\Gamma$ a graph, $\Sigma_n$-equivariantness prevents Farley-Sabalka's gradient field on $\DF(\Gamma,n)$ to be optimal, as the number of critical cells in any given dimension is always divisible by~$n!$ (After all, Farley-Sabalka's gradient field was introduced as a tool to study \emph{full} graph braid groups.) In contrast, $W(\Gamma,n)$ happens to be optimal in many instances, including complete graphs~$\Delta^{m,1}$. As a result, (co)homological calculations for $\DF(\Gamma,n)$ are far more accessible via our gradient field than via Farley-Sabalka's. In particular, just as Farley-Sabalka's gradient field has been a key tool for characterizing in~\cite{MR2355034,MR2833585,MR3426912} \emph{full} braid groups that are right-angled Artin groups (RAAG's), we believe that $W(\Gamma,n)$ will lay an effective strategy in an eventual RAAG-characterization in the case of \emph{pure} graph braid groups ---Ghrist's original question in~\cite{MR1873106}.
}\end{remark}

All in all, the use of our discrete gradient field $W(X,n)$ opens up a gateway for a systematic study, both theoretical as well as computer-assisted, of the algebraic topology properties of general discrete configuration spaces $\DF(X,n)$, their eventual comparison to $\F(X,n)$, and potential applications (e.g., the one noted at the end of Remark~\ref{raaghrist}).

\section{Review of discrete Morse theory}\label{subsectiondmthry}
We start by reviewing the basics on Forman's discrete Morse theory for the type of cell complexes we are interested in. The reader is referred to~\cite{MR1358614,MR1926850} for details.

\medskip
Let $X$ be a finite regular CW complex, and let $F$ denote the set of cells of $X$ (also called ``faces'', even if $X$ fails to be simplicial). We write $\alpha^{(p)}$ to indicate that $\alpha\in F$ is $p$-dimensional. The directed Hasse diagram $H$ of $X$ is the directed graph with vertex set $F$ and directed edges $$\alpha^{(p+1)}\searrow\beta^{(p)}$$ whenever $\beta\subset\alpha$ with $p\geq0$. Recall that the degree of a vertex $v$ in a directed graph is the number of outgoing edges from $v$ plus the number of ingoing edges to $v$.

\begin{definition}
\begin{enumerate}
\item In the context above, a directed subgraph $W$ of $H$ is called a \emph{discrete field} provided all vertices of $W$ have degree~1. The directed graph obtained from $H$ by reversing the orientation of all arrows in $W$ defines the \emph{$W$-modified Hasse diagram} $H_W$. We use the notation $\beta^{(p)}\nearrow\alpha^{(p+1)}$ to indicate the edge in $H_W$ resulting from reversing the orientation of an edge $\alpha^{(p+1)} \searrow \beta^{(p)}$ in $W$. In such a case,~$\alpha$ is said to be \emph{collapsible} while $\beta$ is said to be \emph{redundant.} Any other vertex of $H_W$, i.e.~those not belonging to $W$, are called \emph{critical.}
\item A \emph{gradient path} $\gamma$ is a chain
$$\alpha_0\nearrow\beta_1\searrow\alpha_1\nearrow\cdots\nearrow \beta_k\searrow\alpha_k$$
of alternating up-going/down-going directed edges in $H_W$ with $k>0$. We say that $\gamma$ is a cycle provided $\alpha_0=\alpha_k$. (By construction, $\alpha_i\neq\alpha_{i+1}$ for $0\leq i<k$, so the cycle condition is possible only with $k>1$.) The discrete field~$W$ is said to be \emph{gradient} if no gradient path is a cycle.
\item A \emph{lower} gradient path $\gamma$ is a chain
$$\alpha_0\searrow\beta_1\nearrow\alpha_1\searrow\cdots\searrow \beta_k\nearrow\alpha_k$$
of alternating down-going/up-going directed edges in $H_W$ with $k>0$.
\end{enumerate} 
\end{definition}

A discrete gradient field $W$ on a finite regular CW complex $X$ should be understood as an organized set of instructions that allow us to simplify the cell structure of $X$ without altering its homotopy type. The simplification is achieved by extrapolating, from the simplicial setting, the concept of ``collapse''. Assume that $\tau^{(p)}$ is a free face of $\sigma^{(p+1)}$ in $X$, i.e., $\sigma$ is the only $(p+1)$-dimensional face of $X$ containing $\tau$. We can then squeeze $\sigma$ towards $\partial\sigma-\text{Int}(\tau)$ by pushing $\tau$ through $\text{Int}(\sigma)$. See Figure~\ref{apachurramiento}, where a linear squeezing is depicted at the level of domains of characteristic maps.
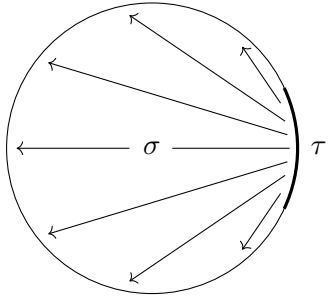
\begin{figure}
\centering
\begin{tikzpicture}[x=.6cm,y=.6cm]
\draw (0,0) ellipse (55pt and 55pt);
\draw(3.05,0)--(.45,0);
\draw[->](-.5,0)--(-3,0);
\draw[->](3,.3)--(-2.3,1.9);
\draw[->](3,-.3)--(-2.3,-1.9);
\draw[->](2.95,.6)--(-.5,2.95);
\draw[->](2.95,-.6)--(-.5,-2.95);
\draw[->](2.85,1)--(2,2.25);
\draw[->](2.85,-1)--(2,-2.25);
\node at (0,3.7) {}; 
\node at (0,0) {$\sigma$}; 
\node at (3.7,0) {$\tau$};
\draw[very thick] (2.93,-1.34) arc (-25:25:1.9cm);
\end{tikzpicture}
\caption{Squeezing of $\sigma$ by pushing through its free face $\tau$ in the regular complex $X$}
\label{apachurramiento}
\end{figure}
Here ``Int'' stands for the interior of a simplex. Since we do not want to change the homotopy type of $X$, the boundary of $\tau$ must stay fixed throughout the squeezing process. On the other hand, at the level of the characteristic-map domain, the squeezing is maximal at the barycenter of~$\tau$. Once the squeezing is complete, we get a new complex, which is homotopy equivalent to $X$, and with the same face structure as that of $X$, except that $\tau$ and $\sigma$ have now been removed. Although a pair of cells $(\tau,\sigma)$ as above would typically be the start of a largest possible gradient path in $X$, the squeezing process can be performed even if $\tau$ is not a free face of~$\sigma$. In such a case, and in order to avoid tearing apart the fabric of $X$, we have to elongate all other cells having $\tau$ as a face, so to match the squeezing into $\sigma$. The process can then be iterated following the instructions provided by each maximal gradient path
$$\alpha_0\nearrow\beta_1\searrow\alpha_1\nearrow\cdots\nearrow \beta_k\searrow\alpha_k,$$ pushing in turn each $\alpha_i$ so to squeeze the corresponding $\beta_{i+1}$. As the process continues, all pushable and all squeezable faces get deformed establishing a flow leading to critical faces. At the end, we get a cell complex which (1) is homotopy equivalent to $X$ and (2) has a $p$-cell for each critical $p$-cell in $X$. Formally:

\begin{theorem}[{\cite[Corollary~3.5]{MR1358614}}]\label{ftdmt}
A discrete gradient field $W$ on a finite regular CW complex~$X$ determines a homotopy equivalence $X\simeq Y$, where the CW complex $Y$ has a $p$-dimensional cell for each $p$-dimensional critical face of $X$.
\end{theorem}

In particular, gradient-path dynamics on $W$ yield a way to compute the homology groups of~$X$. Explicitly, start by choosing an arbitrary orientation for each cell of $X$. For faces $\alpha^{(p)}\subset \beta^{(p+1)}$ of~$X$, let $\iota_{\alpha,\beta}$ denote the incidence number of $\alpha$ and $\beta$, i.e., the coefficient of $\alpha$ in the expression of the cellular boundary $\partial(\beta)$. Note that $\iota_{\alpha,\beta}=\pm1$, in view of the regularity assumption on $X$. In these terms, the multiplicity of a gradient path $\lambda=\left(\alpha_0\nearrow\beta_1\searrow\alpha_1\nearrow\cdots\nearrow \beta_k\searrow\alpha_k\right)$ is
\begin{equation}\label{sgns}
\mu(\lambda):=\prod_{i=1}^k\left(-\iota_{\alpha_{i-1},\beta_i}\cdot\iota_{\alpha_i,\beta_i}\rule{0mm}{3.6mm}\right).
\end{equation}
Then, the homology of $X$ with coefficients in a commutative unital ring $R$ can be computed through the \emph{Morse chain complex} $(\mu_*(X;R),\partial)$, which is degree-wise $R$-free, with basis in dimension $p\geq0$ given by the set of oriented $p$-dimensional critical faces of~$X$. The \emph{Morse boundary} is given at a critical cell $\beta^{(p+1)}$ by
\begin{equation}\label{morseboundary}
\partial(\beta)\,=\!\!\sum_{\text{\tiny $\gamma^{(p)}$ critical}} \iota^{\mu}_{\gamma,\beta}\,\cdot\,\gamma.
\end{equation}
Here, the Morse-type incidence number $\iota^{\mu}_{\gamma,\beta}$ is given by
\begin{equation}\label{incidencealamorse}
\iota^{\mu}_{\gamma,\beta}=\sum_{\alpha^{(p)}} \iota_{\alpha,\beta}\left(\sum_{\lambda}\mu(\lambda)\right)\cdot\gamma,
\end{equation}
where the outer summation runs over all faces $\alpha$ of $\beta$, and the inner summation runs over all gradient paths $\lambda$ from $\alpha$ to $\gamma$.

\section{The gradient field}\label{optgrafie}
The regular complexes we are interested in arise as subcomplexes of powers of simplicial complexes.\footnote{The methods in this section can be extended to subcomplexes of cartesian products of simplicial complexes; details will be given elsewhere.} We fix notation on such a context. Let $K$ denote a finite (abstract) simplicial complex, with set of simplices~$F$ and with a chosen linear ordering $\preceq$ on its vertex set $V\!$. Without loss of generality, we may as well assume $$V=\{1,2,\ldots,m\},$$ with $\preceq$ corresponding to the usual ordering of integers. As above, we write $\alpha^{(p)}$ to indicate that $\alpha\in F$ is $p$-dimensional, and use the more specific list-type notation $$\alpha=\alpha_0\alpha_1\cdots\alpha_p,$$ where $\alpha$ is oriented according to the ordered listing $$\alpha_0<\alpha_1<\cdots<\alpha_p$$ of its vertices. In these terms, the set of $p$-dimensional simplices is ordered lexicographically. Namely,
$$\alpha_0\alpha_1\cdots\alpha_p<\beta_0\beta_1\cdots\beta_p\Longleftrightarrow  \text{ there is }\ell\in\{0,1,\ldots,p\} \text{ with } \alpha_\ell<\beta_\ell \text{ but } \alpha_i=\beta_i \text{ for } 0\leq i<\ell.$$ The latter order is then extended to all of $F$ by declaring that any $p$-dimensional simplex is smaller than any $q$-dimensional simplex whenever $p<q$.

\medskip
Next, take the product cell decomposition in the $n$-th power $K^n$ of $K$. Thus, cells of $K^n$ are prism-like products of simplices $\alpha=\alpha_1\times\alpha_2\times\cdots\times\alpha_n$, with each $\alpha_i$ a simplex of $K$. It will be convenient to use the notation $\alpha=(\alpha_1,\alpha_2,\ldots,\alpha_n)$ or, more explicitly,
\begin{equation}\label{productofsimplices}
\alpha=\left(
\alpha_{1,0}\alpha_{1,1}\cdots\alpha_{1,p_1}\hspace{.3mm},\hspace{.5mm}\alpha_{2,0}\alpha_{2,1}\cdots\alpha_{2,p_2}\hspace{.3mm},\hspace{.5mm}\cdots,\hspace{.5mm}
\alpha_{n,0}\alpha_{n,1}\cdots\alpha_{n,p_n}
\rule{0mm}{4mm}\right).
\end{equation}
For $1\leq i\leq n$ and $0\leq j\leq p_i$, the $(i,j)$-th codimension-1 face of $\alpha$, denoted by $\partial_{i,j}(\alpha)$, is obtained by removing the $j$-th vertex from the $i$-th coordinate of~(\ref{productofsimplices}). Explicitly,
$$
\partial_{i,j}(\alpha)=\left(
\alpha_{1,0}\cdots\alpha_{1,p_1}\hspace{.3mm},\hspace{.7mm}\cdots,\hspace{.5mm}
\alpha_{i,0}\cdots\widehat{\alpha_{i,j}}\cdots\alpha_{i,p_i}\hspace{.3mm},\hspace{.7mm}\cdots,\hspace{.5mm}
\alpha_{n,0}\alpha_{n,1}\cdots\alpha_{n,p_n}
\rule{0mm}{4mm}\right).
$$
For instance, with the chosen orientations, the incidence number of $\alpha$ and $\partial_{i,j}(\alpha)$ is given by 
$$[\alpha\colon\partial_{i,j}(\alpha)]=(-1)^{\epsilon_{i,j}},$$ where $\epsilon_{i,j}=p_1+\cdots+p_{i-1}+j.$ Lastly, the lexicographic order on $F$ defined at the end of the previous paragraph is extended, again lexicographically over the $n$ coordinates, to all cells~(\ref{productofsimplices}).

\medskip
The goal in this section is the construction of a discrete gradient field $W(K,n)$ on the subcomplex $\DF(K,n)$ of $K^n$. Although it is not explicit from the notation, $W(K,n)$ certainly depends on~$\preceq\hspace{.3mm}$, that is, on the chosen ordering of vertices leading to the identification of the vertex set of $K$ with $\{1,2,\ldots,m\}$.

\medskip
By definition, the $(p_1+\cdots+p_n)$-dimensional cell $\alpha$ in~(\ref{productofsimplices}) belongs to~$\DF(K,n)$ precisely when the $n$ sets of vertices 
\begin{equation}\label{sets}
\{\alpha_{1,0},\ldots,\alpha_{1,p_1}\},\ldots,\{\alpha_{n,0},\ldots,\alpha_{n,p_n}\}
\end{equation}
are pairwise disjoint. Let $V(\alpha)$ stand for the union of the sets in~(\ref{sets}). For $i\in\{1,2,\ldots,n\}$ and a vertex $v\not\in V(\alpha)$ such that $\{\alpha_{i,0},\ldots,\alpha_{i,p_i}\}\cup\{v\}$ is a simplex of $K$, we let
\begin{equation}\label{suma}
\alpha+_iv
\end{equation}
stand for the cell of $\DF(K,n)$ obtained from~(\ref{productofsimplices}) by inserting the vertex $v$ in the list of vertices in the $i$-th entry of $\alpha$. Of course, in terms of the list-type notation above, insertion of $v$ is made at a position that renders an ordered list of vertices. Accordingly, the face $\partial_{i,j}(\alpha)$ will also be denoted as
\begin{equation}\label{resta}
\alpha-_i\alpha_{i,j}.
\end{equation}
We will need to iterate the notation in~(\ref{suma}) and~(\ref{resta}). For instance, we will write $\alpha-_iu+_jv$ as a shorthand of $(\alpha-_iu)+_jv$, whenever the constructions make sense.

\subsection{Algorithmic construction}
With the preparation above, set $d:=\dim(X)$ where $X:=\DF(K,n)$. The field $W:=W(K,n)$ is constructed through an algorithmic procedure initialized by $W:=\varnothing$, together with the initialization of auxiliary variables $A^p:=F^p$, for $0\leq p\leq d$. Here~$F^p$ stands for the set of $p$-dimensional cells of $\DF(K,n)$ with the ordering inherited from the cells of~$K^n$ discussed at the end of the second paragraph of Section~\ref{optgrafie}. Throughout the construction of~$W$, we will remove elements from each~$A^p$, keeping the restricted order on the resulting sets $A^p$. The role of $A^p$ is to keep track of the $p$-dimensional faces of $\DF(K,n)$ that are \emph{available} at any given stage of the algorithm, i.e., $p$-dimensional faces of $\DF(K,n)$ that, at the given stage, are not part of any (so far constructed) directed edge in~$W$. Throughout the algorithm, new directed edges $$\alpha^{(p)}\nearrow\beta^{(p+1)},$$ with $(\alpha,\beta)\in A^p\times A^{p+1}$, are added to $W$ by means of a family of processes $\mathcal{P}(p)$ executed for $p=d-1,d-2,\ldots,1,0$ (in that order). Process $\mathcal{P}(p)$, which is executed provided  the current sets of available faces $A^p$ and $A^{p+1}$ are not empty (so there is a chance to add new directed edges to~$W$), consists of a family of subprocesses $\mathcal{P}(p,r)$ executed for $r=n,n-1,\ldots,1$ (in that order). The goal of $\mathcal{P}(p,r)$ is to add to $W$ directed edges of the form
\begin{equation}\label{available}
\alpha^{(p)}\nearrow\alpha^{(p)}+_ru,
\end{equation}
for some vertex $u\in\{1,2,\ldots,m\}$. To accomplish that, $\mathcal{P}(p,r)$ consists of a family of subprocesses $\mathcal{P}(p,r,u)$ executed for $u=m,m-1,\ldots,1$ (in that order), while each $\mathcal{P}(p,r,u)$ consists of a family of instructions $\mathcal{P}(p,r,u,\alpha)$ executed for $\alpha$ in (the current) $A^p$ following its lexicographic order, though this time from the smaller cell to the larger cell. At the core of the nested family of processes, $\mathcal{P}(p,r,u,\alpha)$ checks whether, at the moment of its execution,
\begin{equation}\label{assessment}
(\alpha,\alpha+_ru)\in A^p\times A^{p+1},
\end{equation}
i.e., whether~(\ref{available}) is available to form a new directed edge. If so, $\mathcal{P}(p,r,u,\alpha)$ adds~(\ref{available}) to $W$ and sets $\alpha$ and $\alpha+_ru$ as no longer available by removing them from $A^p$ and $A^{p+1}$, respectively.

\smallskip
From its construction we see that the family $W$  of directed edges resulting at the end of the algorithm forms a discrete field. Furthermore, if $\alpha^{(p)}$ is a critical cell, i.e., a cell in a final $A^p$, then no face $\beta^{(p-1)}$ and no coface $\gamma^{(p+1)}$ of $\alpha$ can be critical. This yields Theorem~\ref{maintheoremmaximal} in view of Proposition~\ref{aciclicidad} below.

\smallskip
In preparation for the arguments in the rest of the paper, the reader will find particularly enlightening the following situation:

\begin{example}\label{ejk3}{\em
Consider the case of the complex $\DF(K_3,2)$ for the complete graph $K_3$. Since $\dim(\DF(K_3,2))=1$, the construction of $W(K_3,2)$ reduces to the process $\mathcal{P}(0)$. At the start of the construction, the ordered list of available faces are $$A^0=\{(1,2),(1,3),(2,1),(2,3),(3,1),(3,2)\}$$ and $$A^1=\{(1,23),(2,13),(3,12),(12,3),(13,2),(23,1)\}.$$ Then $\mathcal{P}(0)$ evolves in the following order:
\begin{itemize}
\item $\mathcal{P}(0,2,3)$ constructs the directed edges $(1,2)\nearrow(1,23)$ and $(2,1)\nearrow(2,13)$ (in that order), which renders $A^0=\{(1,3),(2,3),(3,1),(3,2)\}$ and $A^1=\{(3,12),(12,3),(13,2),(23,1)\}$.
\item $\mathcal{P}(0,2,2)$ constructs the directed edge $(3,1)\nearrow(3,12)$, which renders $A^0=\{(1,3),(2,3),(3,2)\}$ and $A^1=\{(12,3),(13,2),(23,1)\}$.
\item $\mathcal{P}(0,1,2)$ constructs the directed edge $(1,3)\nearrow(12,3)$, which renders $A^0=\{(2,3),(3,2)\}$ and $A^1=\{(13,2),(23,1)\}$.
\item $\mathcal{P}(0,1,1)$ constructs the directed edge $(3,2)\nearrow(13,2)$, which renders the critical cells $(2,3)$ and $(23,1)$ in dimensions 0 and 1, respectively. This is of course compatible with the well known homotopy equivalence $\F(S^1,2)\simeq S^1$. 
\end{itemize}
}\end{example}

\subsection{Acyclicity}

\begin{lemma}\label{redforzado}
Assume $\alpha^{(p)}\nearrow\beta=\alpha+_iv$ lies in $W$ and take $\gamma=\beta-_jw$. If $i<j$ or, else, if $i=j$ with $v<w$, then $\gamma$ is redundant with $\gamma\nearrow\delta=\gamma+_ku$, where either $j<k$ or, else $j=k$ and $w<u$.
\end{lemma}
\begin{proof}
In the algorithm, the instruction $\mathcal{P}(p,j,w,\gamma)$, which assesses the availability of the pair $(\gamma,\beta)$, is executed before $\mathcal{P}(p,i,v,\alpha)$, which is responsible for the construction of the assumed $\alpha\nearrow\beta$. This means that, although $\beta$ is available at the moment of executing $\mathcal{P}(p,j,w,\gamma)$, $\gamma$ is not. So $\gamma$ must have been identified as redundant through in instruction $\mathcal{P}(p,k,u,\gamma)$ previous to $\mathcal{P}(p,j,w,\gamma)$. This yields the result.
\end{proof}

The crux of the matter in the previous proof is the observation that, in the algorithmic construction of $W$, the directed edge $\gamma\nearrow\delta$ must have already been constructed by the time of the construction of $\alpha\nearrow\beta$. Such a ``timing'' issue will be crucial for later arguments in the paper.

\begin{proposition}\label{aciclicidad}
The field $W$ is gradient.
\end{proposition}
\begin{proof}
Assume, for a contradiction, a gradient cycle in $H_W$
$$
\alpha_0\nearrow\beta_0\searrow\alpha_1\nearrow\beta_1\searrow\cdots\searrow\alpha_n\nearrow\beta_n\searrow\alpha_0\,,
$$
with $|\{\alpha_0,\alpha_1,\ldots,\alpha_n\}|=n+1$ (so $n\geq1$), and set $$\alpha_{n+1}:=\alpha_0 \text{ \ \ \ and \ \ \  } \beta_{n+1}:=\beta_0.$$ We can arrange the start of the cycle so to assume that $\alpha_0\nearrow\beta_0$ is constructed before any other $\alpha_r\nearrow\beta_r$, $1\leq r\leq n$. Say $\alpha_r\nearrow\beta_r=\alpha_r+_{i_r}v_r$, for $0\leq r\leq n$, so that
\begin{itemize}
\item[(i)] $i_r\leq i_0$,
\item[(ii)] $i_r=i_0 \;\;\Rightarrow \;\;v_r\leq v_0$,
\end{itemize}
for $1\leq r\leq n$. Let $t$ be the first integer in $\{1,2,\ldots,n+1\}$ so that $v_0\not\in V(\alpha_t)$ (this uses $\alpha_0=\alpha_{n+1}$ and $v_0\not\in V(\alpha_0)$). Note that $t\geq2$, as $\alpha_0\neq\alpha_1$. So, $v_0$ is kept as a vertex in the $i_0$-th entry of each~$\alpha_\ell$ with $1\leq \ell\leq t-1$, while
\begin{itemize}
\item[(iii)] $\alpha_{t-1}\nearrow\beta_{t-1}=\alpha_{t-1}+_{i_{t-1}}v_{t-1}$, with $v_{t-1}\neq v_0$,
\item[(iv)] $\alpha_t=\beta_{t-1}-_{i_0}v_0$ \hspace{.5mm}(in particular $t\leq n$).
\end{itemize}
However, in view of conditions (i) and (ii) for $r=t-1$ and $r=t$, the conclusion of Lemma~\ref{redforzado} applied to the situation in conditions~(iii) y (iv) yields $v_{t-1}=v_0$, which contradicts (iii).
\end{proof}

\section{Classification of cells}\label{crigradflo}
In the rest of the paper we focus on the case of $\DF(\Delta^{m,d},2)$, assuming $1\leq d\leq m-2$ so to avoid the two exceptional but well-known cases $\F(\Delta^m,2)\simeq S^{m-1}\simeq\F(S^{m-1},2)$. For exposition purposes, it is convenient to make the following notational adjustments, which will be in force from now on. Firstly, the simplex under consideration is thought of as~$\Delta^{\mu-1}$, with vertex set $\{1,2,\ldots,\mu\}$, and the skeleton under consideration is~$\Delta^{\mu-1,\mu-\kappa-1}$. Thus
\begin{equation}\label{initialbounds}
\mbox{$2\leq \kappa\leq \mu-2$ \ (so $4\leq\mu$).}
\end{equation}
Secondly, following the notation of~(\ref{productofsimplices}), cells of $X_{\mu,\kappa}:=\DF(\Delta^{\mu-1,\mu-\kappa-1},2)$ are written in the form
$(A,B):=(a_1a_2\cdots a_r,b_1b_2\cdots b_s)$, with
\begin{align}
&1\leq a_1<a_2<\cdots<a_r\leq\mu, \;\;\;\; 1\leq r\leq \mu-\kappa,\nonumber\\
&1\leq \hspace{.35mm}b_1<\hspace{.35mm}b_2<\cdots<\hspace{.35mm}b_s\leq\mu, \;\;\;\; 1\leq s\leq \mu-\kappa,\label{acotado}\\
&\{a_1,a_2,\ldots,a_r\}\cap\{b_1,b_2,\ldots,b_s\}=\varnothing.\nonumber
\end{align}
If we need to refer to the sets $\{a_1,a_2,\ldots,a_r\}$ or $\{b_1,b_2,\ldots,b_s\}$ (as in the third condition above), we write $\{A\}$ or $\{B\}$. Also, when dealing with multiple cells of $X_{\mu,\kappa}$ at once, we use a (multi-)primed notation, e.g.,~$(A',B')$ with ingredients $a'_i$, $b'_i$, $r'$ and $s'$. Lastly, recall from the introduction that $\langle p,q\rangle$ stands for the set of integers $\{p,p+1,\ldots,q-1,q\}$.

\medskip
Note that $\dim(A,B)=r+s-2\leq \min\{\mu-2,2(\mu-\kappa-1)\}$. In fact:

\begin{proposition}\label{pure}
All maximal cells of $X_{\mu,\kappa}$ have dimension $\min\{\mu-2,2(\mu-\kappa-1)\}$, i.e., $X_{\mu,\kappa}$ is \emph{pure} of degree $\min\{\mu-2,2(\mu-\kappa-1)\}$.
\end{proposition}
\begin{proof}
We give the straightforward argument for completeness. The situation for $\mu\leq2\kappa$, i.e., when $\min\{\mu-2,2(\mu-\kappa-1)\}=2(\mu-\kappa-1)$, is easier, as any cell $(A,B)$ can then be realized as a face of a cell $(A',B')$ with $r'=s'=\mu-\kappa$. So, assume $2\kappa\leq\mu$, i.e., $\min\{\mu-2,2(\mu-\kappa-1)\}=\mu-2$.

\medskip\noindent{\bf Case  $\kappa\leq r$}. We can add the $\mu-r-s$ vertices in $\langle1,\mu\rangle-V(A,B)$ into the $B$-entry, so to realize $(A,B)$ as a face of a cell $(A,B')$, as $|B'|=s+\mu-r-s\leq\mu-\kappa$. 

\medskip\noindent{\bf Case  $\kappa\leq s$}. We can add the $\mu-r-s$ vertices in $\langle1,\mu\rangle-V(A,B)$ into the $A$-entry, so to realize $(A,B)$ as a face of a cell $(A',B)$, as $|A'|=r+\mu-r-s\leq\mu-\kappa$.

\medskip\noindent{\bf Case $r<\kappa>s$}. Then there are $(\kappa-r)+(\kappa-s)$ vertices in $\langle1,\mu\rangle-V(A,B)$, and we add $\kappa-r$ (respectively, $\kappa-s$) of them into the $A$-entry (respectively, $B$-entry), so to realize $(A,B)$ as a face of a cell $(A',B')$, with $r'\geq\kappa\leq s'$. We are then done in view of any of the two previous cases.
\end{proof}

Let $W_{\mu,\kappa}:=W(\Delta^{\mu-1,\mu-\kappa-1},2)$. We describe in Theorem~\ref{teorema1} below all directed edges in $W_{\mu.\kappa}$ of the form $(A,B)\nearrow(A,B)+_2u$. The corresponding description for the totality of directed edges in~$W_{\mu.\kappa}$ of the form $(A,B)\nearrow(A,B)+_1u$ is then given in Theorem~\ref{teorema2}.

\begin{theorem}\label{teorema1}
Any cell $(A,B)$ of $X_{\mu,k}$ with $s>1$ and
\begin{equation}\label{conten}
\langle b_s+1,\mu\rangle\subseteq\{A\}
\end{equation}
is collapsible. In fact
\begin{equation}\label{insercion1}
(A,B)-_2b_s\nearrow(A,B)
\end{equation}
is a directed edge in $W_{\mu,\kappa}$. Furthermore, the directed edges in~(\ref{insercion1}) account for all directed edges in $W_{\mu,\kappa}$ that arise by insertion of a vertex on the second coordinate of a cell in~$X_{\mu,\kappa}$.
\end{theorem}

Of course, condition~(\ref{conten}) might hold just because $\langle b_s+1,\mu\rangle=\varnothing$, i.e., $b_s=\mu$.

\begin{proof}
Recall $\dim(A,B)=r+s-2$. We show below that $(A,B)$ is not redundant, so that it is available at the start of process $\mathcal{P}(r+s-3)$. Then $\mathcal{P}(r+s-3,2,b_s)$ is the first subprocess that could make $(A,B)$ lose the availability condition, namely, with $\mathcal{P}(r+s-3,2,b_s,(A,B)-_2b_s)$ assessing (as in~(\ref{assessment})) the pair $$\left(\rule{0mm}{3.5mm}(A,B)-_2b_s\,,\,(A,B)\right).$$ By~(\ref{conten}), $(A,B)-_2b_s$ is available at the moment of executing $\mathcal{P}(r+s-3,2,b_s,(A,B)-_2b_s)$, so the latter instruction does construct~(\ref{insercion1}).

\smallskip
To complete the proof of the first part of the theorem, it remains to show that $(A,B)$ is not redundant. Assume for a contradiction that a directed edge
\begin{equation}\label{imposible1}
(A,B)\nearrow(A,B)+_ju
\end{equation}
belongs to $W_{\mu,\kappa}$ for $j\in\{1,2\}$. Note that condition~(\ref{conten}) forces $u<b_s$. Therefore, before~(\ref{imposible1}) is constructed by $\mathcal{P}(r+s-2,j,u,(A,B))$, the earlier instruction $\mathcal{P}(r+s-2,2,b_s,(A,B)+_ju-_2b_s)$ is in charge of assessing the pair $$\left((A,B)+_ju-_2b_s,(A,B)+_ju\rule{0mm}{3.5mm}\right).$$ In such an assessment, both coordinates are available: the second one because it stays available up to $\mathcal{P}(r+s-2,j,u,(A,B))$, and the first one in view of~(\ref{conten}). So $\mathcal{P}(r+s-2,2,b_s,(A,B)+_ju-_2b_s)$ constructs $(A,B)+_ju-_2b_s\nearrow(A,B)+_ju$, which contradicts~(\ref{imposible1}).

\smallskip
To complete the proof, assume ---again for a contradiction--- that there is a directed edge
\begin{equation}\label{imposible2}
(A',B')-_2b'_i\nearrow(A',B')
\end{equation}
in $W_{\mu-\kappa}$ with either $i<s'$ or so that the corresponding (primed) condition~(\ref{conten}) is not fulfilled. In the latter case set $u_0:=\max\{u\colon b'_{s'}<u\leq\mu \text{ and } u\not\in A'\}$, so that $(A',B')-_2b'_i\nearrow(A',B')-_2b'_i+_2u_0$ has already been proved to lie in $W_{\mu,\kappa}$, which prevents~(\ref{imposible2}). Therefore, we can assume $i<s'$ with $\{u\colon b'_{s'}<u\leq\mu \text{ and } u\not\in A'\}=\varnothing$. But then $(A',B')-_2b'_s\nearrow(A',B')$ has already been proved to lie in $W_{\mu,\kappa}$, again preventing~(\ref{imposible2}).
\end{proof}

\begin{theorem}\label{teorema2}
Let $(A,B)$ be a cell of $X_{\mu,\kappa}$ with $r>1$.
\begin{enumerate}[(1)]
\item Assume~(\ref{conten}) holds and $s=1$. If $b_1-1\in\{A\}$, then $(A,B)-_1(b_1-1)\nearrow(A,B)$ is a directed edge in $W_{\mu,\kappa}$.
\item Assume~(\ref{conten}) fails (in particular $b_s<\mu$) and $s=\mu-\kappa$.
\begin{enumerate}
\item If $a_r=\mu$, then $(A,B)-_1\mu\nearrow(A,B)$ is a directed edge in $W_{\mu,\kappa}$.
\item If $a_r<\mu$ and
\begin{itemize}
\item[] $ (b.1) \;\; \langle b_s+1,\mu-1\rangle\subseteq\{A\};$\vspace{.5mm}
\item[] $ (b.2) \;\; \langle1,b_s-1\rangle\not\subseteq\{B\};$
\item[] $ (b.3) \;\; \overline{a}:=\max\left(\langle1,b_s-1\rangle-\{B\}\rule{0mm}{3.4mm}\right)\in\{A\},$
\end{itemize}
then $(A,B)-_1\overline{a}\nearrow(A,B)$ is a directed edge in $W_{\mu,\kappa}$.
\end{enumerate}
\end{enumerate}
Furthermore, the three types of directed edges described above account for all directed edges in $W_{\mu,\kappa}$ that arise by insertion of a vertex on the first coordinate of a cell in~$X_{\mu,\kappa}$.
\end{theorem}

Note that, if $b_s<\mu-1$ in $(b)$, then $(b.1)$ forces $a_r=\mu-1$. Yet, it might just as well be the case that $b_s=\mu-1$ in item $(b)$, in which case $(b.1)$ is vacuously true.

\begin{proof}
Let $(A,B)$ be a cell in $X_{\mu,\kappa}$ satisfying $r>1$ together with the conditions in item \emph{(1)} of the theorem. In particular $B=b$ with $b=b_1$. First, we assert that~$(A,b)$ is not redundant. Indeed, no directed edge $(A,b)\nearrow(A,b)+_2u$ can lie in $W_{\mu,\kappa}$, in view of Theorem~\ref{teorema1} (as $u<b-1$ is forced). On the other hand, before some directed edge
\begin{equation}\label{imposible3}
(A,b)\nearrow(A,b)+_1u
\end{equation}
can be constructed, necessarily by $\mathcal{P}(r-1,1,u,(A,b))$ and with $u<b-1$ (the latter condition in view of the hypotheses in~\emph{(1)}), the earlier instruction $\mathcal{P}(r-1,1,b-1,(A,b)+_1u-_1(b-1))$ would have to asses the pair
$$
\left(\rule{0mm}{3.5mm} (A,b)+_1u-_1(b-1)\,,\,(A,b)+_1u\right),
$$
both of whose entries are found to be available at that moment: the second one because it is supposed to stay available up to $\mathcal{P}(r-1,1,u,(A,b))$, and the first one in view of Theorem~\ref{teorema1}, which rules out a possible insertion of a vertex in the second entry, and because of the hypotheses in~\emph{(1),} which rule out a possible earlier insertion of a vertex in the first entry. This prevents~(\ref{imposible3}) and completes the argument for the non-redundancy of $(A,b)$. Then, in order to prove that the directed edge $(A,b)-_1(b-1)\nearrow(A,b)$ lies in $W_{\mu,\kappa}$, it remains to show that $(A,b)-_1(b-1)$ and $(A,b)$ remain available, the former one as non-redundant and the latter one as non-collapsible, up to the moment of executing $\mathcal{P}(r-2,1,b-1,(A,b)-_1(b-1))$. Taking into account the current hypotheses, the task amounts to ruling out the following two types of directed edges from lying in $W_{\mu,\kappa}$:
\begin{itemize}
\item $(A,b)-_1w\nearrow(A,b)$, with $b<w$;
\item $(A,b)-_1(b-1)\nearrow(A,b)-_1(b-1)+_2w$, with $w<b$.
\end{itemize}
Theorem~\ref{teorema1} accounts for the latter option. On the other hand, the former option cannot hold because, for $b<w\in\{A\}$, the directed edge $(A,b)-_1w\nearrow(A,b)-_1w+_2w$ lies in~$W_{\mu,\kappa}$, in view of Theorem~\ref{teorema1}. This completes the proof of item \emph{(1)} of the theorem.

\medskip
For the rest of the proof, let $(A,B)$ with $r>1$ and $s=\mu-\kappa$, but so that~(\ref{conten}) fails (in particular $b_s<\mu$).

\smallskip\noindent 
{\bf Case $a_r=\mu$}. We show below that $(A,B)$ is not redundant. Furthermore, Theorem~\ref{teorema1} rules out a directed edge of the form $(A,B)-_2b_i\nearrow(A,B)$ from lying in $W_{\mu,\kappa}$. Therefore, the first instruction that could make $(A,B)$ lose its availability condition is $\mathcal{P}(r+s-3,1,\mu,(A,B)-_1\mu)$, which assesses the pair $\left((A,B)-_1\mu,(A,B)\right)$ and constructs the asserted directed edge $(A,B)-_1\mu\nearrow(A,B)$. Indeed, just as $(A,B)$, $(A,B)-_1\mu$ is available at the moment of executing $\mathcal{P}(r+s-3,1,\mu,(A,B)-_1\mu)$ because of the hypothesis $s=\mu-\kappa$. It thus remains to show that $(A,B)$ is not redundant. Assume, for a contradiction, that a directed edge
\begin{equation}\label{imposible4}
(A,B)\nearrow(A,B)+_ju
\end{equation}
lies in $W_{\mu,\kappa}$. The hypotheses $s=\mu-\kappa$ and $a_r=\mu$ force $j=1$ and $u<\mu$, respectively. As in similar arguments above, the contradiction comes from observing that, before~(\ref{imposible4}) is constructed by $\mathcal{P}(r+s-2,1,u,(A,B))$, the earlier instruction $\mathcal{P}(r+s-2,1,\mu,(A,B)+_1u-_1\mu)$, which is in charge of assessing the pair
\begin{equation}\label{imposible5}
\left(\rule{0mm}{3.5mm} (A,B)+_1u-_1\mu\,,\,(A,B)+_1u \right),
\end{equation}
constructs the directed edge $(A,B)+_1u-_1\mu\nearrow(A,B)+_1u$. Indeed, at the moment of executing $\mathcal{P}(r+s-2,1,\mu,(A,B)+_1u-_1\mu)$, the second entry in~(\ref{imposible5}) is available,  as it actually stays available up to $\mathcal{P}(r+s-2,1,u,(A,B))$, while the hypothesis $s=\mu-\kappa$ gives the corresponding availability of the first entry in~(\ref{imposible5}). This completes the proof of item \emph{(a)} of the theorem.

\medskip\noindent 
{\bf Case $a_r<\mu$ with $(b.1)$--$(b.3)$ holding}. Start by noticing that, under the current hypotheses, Theorem~\ref{teorema1} implies that
\begin{equation}\label{sucede}
\mbox{\emph{the directed edge $\,(A,B)+_1\mu-_2b_s\nearrow(A,B)+_1\mu\,$ lies in $\,W_{\mu,\kappa}$.}}
\end{equation}
(Note that $s=\mu-\kappa>1$, in view of~(\ref{initialbounds})). Next, we prove that $(A,B)$ is not redundant. Suppose, for a contradiction, that a directed edge
\begin{equation}\label{imposible6}
(A,B)\nearrow((A,B)+_ju)
\end{equation}
lies in $W_{\mu,\kappa}$. The hypothesis $s=\mu-\kappa$ forces $j=1$, while $(b.1)$--$(b.3)$ force either $u=\mu$ or~$u<\overline{a}$. The contradiction then comes from observing that option $u=\mu$ is impossible in view of~(\ref{sucede}), while option $u<\overline{a}$ is impossible for, in such a case, before $\mathcal{P}(r+s-2,1,u,(A,B))$ can construct~(\ref{imposible6}), the earlier instruction $\mathcal{P}(r+s-2,1,\overline{a},(A,B)+_1u-_1\overline{a})$, which assesses the pair
\begin{equation}\label{imposible7}
\left( (A,B)+_1u-_1\overline{a}\,,\, (A,B)+_1u\rule{0mm}{3.5mm}\right),
\end{equation}
constructs the directed edge $(A,B)+_1u-_1\overline{a}\nearrow(A,B)+_1u$. Indeed, at the moment of executing $\mathcal{P}(r+s-2,1,\overline{a},(A,B)+_1u-_1\overline{a})$, the second entry in~(\ref{imposible7}) is available, as it is supposed to stay available up to $\mathcal{P}(r+s-2,1,u,(A,B))$. Likewise, in view of the hypothesis $s=\mu-\kappa$, the only way in which the corresponding availability of the first entry in~(\ref{imposible7}) could fail is with the directed edge $(A,B)+_1u-_1\overline{a}\nearrow (A,B)+_1u-_1\overline{a}+_1\mu$ lying in $W_{\mu,\kappa}$. But such a possibility is ruled out by Theorem~\ref{teorema1}, which shows that, just as~(\ref{sucede}), the directed edge $$(A,B)+_1u-_1\overline{a}+_1\mu-_2b_2\nearrow (A,B)+_1u-_1\overline{a}+_1\mu$$ lies in $W_{\mu,\kappa}$. This completes the argument for the non-redundancy of $(A,B)$. Then, the proof of item~\emph{(b)} of the theorem will be complete as soon as we argue that $(A,B)-_1\overline{a}$ and $(A,B)$ are available, the former one as non-redundant and the latter one as non-collapsible, at the moment of executing the instruction $\mathcal{P}(r+s-3,1,\overline{a},(A,B)-_1\overline{a})$. Taking into account Theorem~\ref{teorema1} and the current hypotheses, the task amounts to ruling out the following two types of directed edges from lying in~$W_{\mu,\kappa}$:
\begin{itemize}
\item $(A,B)-_1w\nearrow(A,B)$, with $b_s<w\in\{a_1,a_2,\ldots,a_r\}$ (in particular $w<\mu$);
\item $(A,B)-_1\overline{a}\nearrow(A,B)-_1\overline{a}+_1\mu$.
\end{itemize}
The former one is ruled out by item \emph{(a)} of the theorem under argumentation, which has already been proved to give $(A,B)-_1w\nearrow(A,B)-_1w+_1\mu\hspace{.4mm}$ as a directed edge lying in $W_{\mu,\kappa}$. The latter one is ruled out by Theorem~\ref{teorema1}, which shows that the directed edge $(A,B)-_1\overline{a}+_1\mu-_2b_s \nearrow (A,B)-_1\overline{a}+_1\mu$ lies in $W_{\mu,\kappa}$. This completes the proof of item \emph{(b)} of the theorem.

\smallskip
The final assertion in this theorem is equivalent to the statement of Proposition~\ref{lascriticas} below, and will be proved after recording two straightforward consequences of the explicit description of directed edges in Theorems~\ref{teorema1} and~\ref{teorema2}.
\end{proof}
  
\begin{corollary}\label{celdasredundantes}
The redundant cells $(A,B)$ associated to the directed edges in Theorems~\ref{teorema1} and~\ref{teorema2} fall in one of the following cases (with labels according to the relevant items in those two theorems): 
\begin{description}
\item[(\ref{teorema1})] If $s<\mu-\kappa$ and~(\ref{conten}) fails, then 
$(A,B)\nearrow(A,B)+_2\max\left(\langle b_s+1,\mu\rangle-\{A\}\rule{0mm}{3.5mm}\right)$ is a directed edge in $W_{\mu,\kappa}$.
\item[(\ref{teorema2}.1)] If $r<\mu-\kappa$, $s=1$, $1\leq b_1-1\not\in\{A\}$ and~(\ref{conten}) holds, then $(A,B)\nearrow(A,B)+_1(b_1-1)$ is a directed edge in $W_{\mu,\kappa}$.
\item[(\ref{teorema2}.2.a)] If $r<\mu-\kappa=s$, $a_r<\mu$ and $\langle b_s+1,\mu-1\rangle\not\subseteq\{A\}$, then $(A,B)\nearrow(A,B)+_1\mu$ is a directed edge in $W_{\mu,\kappa}$. (Note that in this case~(\ref{conten}) fails.)
\item[(\ref{teorema2}.2.b)] If $r<\mu-\kappa=s$, $\max\{a_r,b_s\}<\mu$ and
\begin{itemize}
\item[] $ (c.1) \;\; \langle b_s+1,\mu-1\rangle\subseteq\{A\}$;
\item[] $ (c.2) \;\; \langle1,b_s-1\rangle\not\subseteq\{B\};$
\item[] $ (c.3) \;\; \overline{a}:=\max\left(\rule{0mm}{3.4mm}\langle1,b_s-1\rangle-\{B\}\right)\not\in\{A\},$\end{itemize}
then $(A,B)\nearrow(A,B)+_1\overline{a}$ is a directed edge in $W_{\mu,\kappa}$. (Note that in this case~(\ref{conten}) fails.)
\end{description}
\end{corollary}

\begin{corollary}\label{celdascriticas}
The cells $(A,B)$ that are not involved in any of the directed edges described in Theorems~\ref{teorema1} and~\ref{teorema2} are listed below.
\begin{enumerate}
\item All cases where~(\ref{conten}) holds have $s=1$ and one of:
\begin{enumerate}[(i)]
\item $1=b_1$.
\item $1\leq b_1-1\in\{A\}$ and $\,r=1$.
\item $1\leq b_1-1\not\in\{A\}$ and $\,r=\mu-\kappa$.
\end{enumerate}
\item All cases where~(\ref{conten}) fails (so $b_s<\mu$) have $s=\mu-\kappa$ and one of:
\begin{enumerate}[(i)]\addtocounter{enumii}{3}
\item $a_r=\mu$ and $\,r=1$.
\item $a_r<\mu$, $\langle b_s+1,\mu-1\rangle\not\subseteq\{A\}$ and $r=\mu-\kappa$.
\item $a_r<\mu$, $\langle b_s+1,\mu-1\rangle\subseteq\{A\}$ and $\,\langle1,b_s-1\rangle\subseteq\{B\}$.
\item $a_r<\mu$, $\langle b_s+1,\mu-1\rangle\subseteq\{A\}$, $\langle1,b_s-1\rangle\not\subseteq\{B\}$ with
$\,\overline{a}:=\max\left(\langle1,b_s-1\rangle-\{B\}\right)\in\{A\}$, and $\,r=1$. 
\item $a_r<\mu$, $\langle b_s+1,\mu-1\rangle\subseteq\{A\}$, $\langle1,b_s-1\rangle\not\subseteq\{B\}$ with
$\,\overline{a}:=\max\left(\langle1,b_s-1\rangle-\{B\}\right)\not\in\{A\}$, and $\,r=\mu-\kappa$.
\end{enumerate}
\end{enumerate}
\end{corollary}

\begin{proposition}\label{lascriticas}
All cells listed in Corollary~\ref{celdascriticas} are critical, i.e., Theorems~\ref{teorema1} and~\ref{teorema2} describe the complete gradient field $W_{\mu,\kappa}$.
\end{proposition}
\begin{proof}
Any directed edge in $W_{\mu,\kappa}$ that could be missing in the descriptions of Theorems~\ref{teorema1} and~\ref{teorema2} would have to involve two of the cells described in Corollary~\ref{celdascriticas}. Therefore, it suffices to show that any $(p-1)$-dimensional face of a $p$-dimensional cell in Corollary~\ref{celdascriticas} is involved in a directed edge of Theorems~\ref{teorema1} and~\ref{teorema2}. The situation for cells in \emph{(i)} and \emph{(ii)} is elementary. Indeed, the only possible cell in \emph{(i)} is $(23\ldots\mu,1)$, which is not a cell of~$X_{\mu,\kappa}$ in view of~(\ref{initialbounds}) and~(\ref{acotado}). On the other hand, the only cell in~\emph{(ii)} is $(\mu-1,\mu)$, which has no proper faces. All other instances are analyzed next on a case-by-case basis.
\begin{itemize}
\item The situation for a cell $(A,b)$ in~\emph{(iii)} follows by noticing that
$$\begin{array}{ll}
(A,b)-_1u\nearrow(A,b)-_1u+_2u,        & \mbox{ for $b<u$, in view of Theorem~\ref{teorema1};} \\
(A,b)-_1u\nearrow(A,b)-_1u+_1(b-1), & \mbox{ for $u<b-1$, in view of Theorem~\ref{teorema2}(1).}
\end{array}$$

\item A cell $(\mu,B)$ in~\emph{(iv)} has $b_s<\mu-1$, as~(\ref{conten}) fails, so the desired property follows by noticing that $(\mu,B)-_2u\nearrow(\mu,B)-_2u+_2(\mu-1),$ in view of Theorem~\ref{teorema1}.

\item The situation for a cell $(A,B)$ in \emph{(v)} follows by noticing that 
$$\begin{array}{ll}
(A,B)-_2u\nearrow(A,B)-_2u+_2\mu,        & \mbox{in view of Theorem~\ref{teorema1};} \\
(A,B)-_1u\nearrow(A,B)-_1u+_1\mu, & \mbox{in view of Theorem~\ref{teorema2}(2.a).}
\end{array}$$

\item The only possible cell in \emph{(vi)} is $(A,B)=((\mu-\kappa+1)\cdots(\mu-1),1\cdots(\mu-\kappa))$, which lies in $X_{\mu,\kappa}$ only when $2\kappa\leq\mu+1$. In such a case, the desired property follows by noticing that
$$\begin{array}{ll}
(A,B)-_2u\nearrow(A,B)-_2u+_2\mu,        & \mbox{in view of Theorem~\ref{teorema1};} \\
(A,B)-_1u\nearrow(A,B)-_1u+_1\mu, & \mbox{in view of Theorem~\ref{teorema2}(2.a).}
\end{array}$$

\item The situation for a cell $(a,B)$ in \emph{(vii)} follows by noticing that $(a,B)-_2u\nearrow(a,B)-_2u+_2\mu$, in view of Theorem~\ref{teorema1}.

\item The situation for a cell $(A,B)$ in \emph{(viii)} follows by noticing that 
$$\begin{array}{ll}
(A,B)-_2u\nearrow(A,B)-_2u+_2\mu,        & \mbox{in view of Theorem~\ref{teorema1};} \\
(A,B)-_1u\nearrow(A,B)-_1u+_1\mu, & \mbox{if $b_s<u$, in view of Theorem~\ref{teorema2}(2.a);} \\
(A,B)-_1u\nearrow(A,B)-_1u+_1\overline{a}, & \mbox{if $u<\overline{a}$, in view of Theorem~\ref{teorema2}(2.b).}
\end{array}\vspace{-.9cm}$$
\end{itemize}
\end{proof}

\subsection{Counting critical cells}
We next take a closer look at the critical cells as listed in Corollary~\ref{celdascriticas}. We have observed that (due to our preparatory assumptions~(\ref{initialbounds}) and ~(\ref{acotado})) there are no critical cells fitting in item \emph{(i)}, and that the only critical cell of type \emph{(ii)} is $(\mu-1,\mu)$, which is the only critical 0-dimensional cell. A similar situation holds for cells of type \emph{(vi)}. Namely, as observed in the previous proof, the only possible such cell is  
\begin{equation}\label{topcell}
((\mu-\kappa+1)(\mu-\kappa+2)\cdots(\mu-1),12\cdots(\mu-\kappa)),
\end{equation}
which lies in dimension $\mu-3$, and which is a cell of $X_{\mu,\kappa}$ if and only if $2\kappa-1\leq\mu$.
 
\begin{proposition}\label{rangoestable}
There is a total of $\hspace{.3mm}2\binom{\mu-1}{\mu-\kappa}\hspace{-.5mm}$ critical cells of types either (iii), (iv) or (vii), all of which have dimension $\mu-\kappa-1$. Furthermore, critical cells of types (v) or~(viii) have dimension $2(\mu-\kappa-1)$ and appear only when $\,\mu<2\kappa-1$.
\end{proposition}
\begin{proof}
Elementary counting gives that there are:
\begin{itemize}
\item $\binom{\mu-1}{\mu-\kappa}$ critical cells of type \emph{(iii);}
\item $\binom{\mu-2}{\mu-\kappa}$ critical cells of type \emph{(iv);}
\item $\binom{\mu-2}{\mu-\kappa-1}$ critical cells of type \emph{(vii)}.
\end{itemize}
This makes a total of $\binom{\mu-1}{\mu-\kappa}+\binom{\mu-2}{\mu-\kappa}+\binom{\mu-2}{\mu-\kappa-1}=\binom{\mu-1}{\mu-\kappa}+\binom{\mu-1}{\mu-\kappa}$ critical cells, all of which clearly have dimension $\mu-\kappa-1$. For the second assertion of the proposition observe that a critical cell $(A,B)$ of type \emph{(v)} or \emph{(viii)} satisfies a strict inclusion
\begin{equation}\label{strinc}
\{A\}\cup\{B\}\subset\langle1,\mu-1\rangle.
\end{equation}
which follows from the condition $\langle b_s+1,\mu-1\rangle\not\subseteq \{A\}$, in the case of \emph{(v)}, and from the condition $\max(\langle1,b_s-1\rangle-\{B\})\not\in\{A\}$, in the case of \emph{(viii)}. The result follows since~(\ref{strinc}) holds only with $2(\mu-\kappa)<\mu-1$.
\end{proof}

\begin{figure}
\begin{center}\begin{tabular}{|c|c|c|} 
\hline \rule{0mm}{4mm}  
type & existence & dimension \\ [2pt] \hline \rule{0mm}{4mm}
\emph{(i)} & never & \\ [2pt] \hline \rule{0mm}{4mm}
\emph{(ii)} & always & 0 \\ [2pt] \hline \rule{0mm}{4mm}
\emph{(iii)} & always & $\mu-\kappa-1$ \\ [2pt] \hline \rule{0mm}{4mm}
\emph{(iv)} & always & $\mu-\kappa-1$ \\ [2pt] \hline \rule{0mm}{4mm}
\emph{(vii)} & always & $\mu-\kappa-1$ \\ [2pt] \hline \rule{0mm}{4mm}
\emph{(v)} & $\mu<2\kappa-1$ & $2(\mu-\kappa-1)$ \\ [2pt] \hline \rule{0mm}{4mm}
\emph{(viii)} & $\mu<2\kappa-1$ & $2(\mu-\kappa-1)$ \\ [2pt] \hline \rule{0mm}{4mm}
\emph{(vi)} & $\mu\geq2\kappa-1$ & $\mu-3$ \\ [2pt] \hline
\end{tabular}\end{center}
\caption{Critical cells by type, dimension and condition granting their existence}
\label{distri}
\end{figure}

\begin{proof}[Proof of Theorem~\ref{maintheoremlarge}]
Recall we have neglected the two elementary cases in item~\emph{(\ref{it2})} so, in terms of the notational adjustments in Section~\ref{crigradflo},~(\ref{initialbounds}) holds. Item \emph{(\ref{it1})} then follows from Theorem~\ref{ftdmt} and the summary in Figure~\ref{distri}. For item~\emph{(\ref{it3})}, where $\kappa=2$, use in addition Proposition~\ref{rangoestable}. Likewise, the numeric hypothesis in item~\emph{(\ref{it4})}, which becomes $\kappa\geq3$ and $2\kappa-1\leq\mu$, yields~(\ref{untop}), while the numeric hypothesis in item~\emph{(\ref{it5})} is $\mu<2\kappa-1$ and leads to~(\ref{mutop}) once we count the number of critical cells in dimension $2(\mu-\kappa-1)$. Counting details are provided by Proposition~\ref{cuenta2d} below.

Apart from the second assertion in item~\emph{(\ref{it4})}, whose proof is given in the next section, the only missing argument is the one about the minimality of the cell structures in~(\ref{untop}) and~(\ref{mutop}), i.e., the precise values of the relevant Betti numbers. In both cases it suffices to show the vanishing of the corresponding Morse differentials~(\ref{morseboundary}). Such a condition follows from sparseness, provided $d<m-3$ in the case of~(\ref{untop}), and provided $d>1$ in the case of~(\ref{mutop}). Non-sparse situations are dealt with next.

The case $1<d=m-3$ in~(\ref{untop}) will follow once we prove the second assertion in item~\emph{(\ref{it4})} ---in the next section. On the other hand, the case $1=d=m-3$ in~(\ref{untop}) corresponds to $\DF(\Delta^{4,1},2)$, in which case the triviality of the Morse differential follows from Abrams' homeomorphism in Remark~\ref{abramsobs}. Lastly, the case $d=1$ in~(\ref{mutop}) follows from Proposition~\ref{Betticompletegraphs} below.
\end{proof}

\begin{proposition}\label{cuenta2d}
For $\mu<2k-1$, the number of critical cells in dimension $2(\mu-\kappa-1)$ is
\begin{equation}\label{fea}
\binom{\mu-1}{\mu-\kappa}\binom{\kappa-1}{\mu-\kappa}-\sum_{b=\kappa}^{\mu-1}\;\sum_{a=b-(\mu-\kappa)}^{b-1}\binom{a-1}{b-\kappa}\binom{a-b+\kappa-1}{\mu-\kappa-b+a}.
\end{equation}
\end{proposition}
\begin{proof}
The cells under consideration are those of types \emph{(v)} and \emph{(viii)} in Corollary~\ref{celdascriticas}. Note that $\binom{\mu-1}{\,\mu-\kappa\,}\binom{\kappa-1}{\,\mu-\kappa\,}$ counts the number of cells $(A,B)$ satisfying the conditions
\begin{equation}\label{universo}
a_r<\mu>b_s \text{ \ and \ } r=\mu-\kappa=s,
\end{equation}
which are common to \emph{(v)} and \emph{(viii)}. The result will follow once we explain how the double summation in~(\ref{fea}) counts the cardinality of the set $\mathcal{C}$ of cells satisfying~(\ref{universo}) without being of type \emph{(v)} or \emph{(viii)}.

Note that cells of $\mathcal{C}$ are described by the conditions
\begin{equation}\label{complemento}
\mbox{$a_r<\mu>b_s$, $\;\;r=\mu-\kappa=s$, $\;\;\langle b_s+1,\mu-1\rangle\subseteq A$}
\end{equation}
 and either $\langle1,b_s\rangle=B$ or, else, $\max(\langle1,b_s\rangle-B)\in A$. However, the hypothesis $\mu<2\kappa-1$ implies that condition $\langle1,b_s\rangle=B$ is incompatible with those in~(\ref{complemento}). Therefore, cells of $\mathcal{C}$ are really described by
 $$
 \mbox{$a_r<\mu>b_s$, $\;\;r=\mu-\kappa=s$, $\;\;\langle b_s+1,\mu-1\rangle\subseteq A$, $\;\;\langle1,b_s\rangle\not\subseteq B$, $\;\;\overline{a}:=\max(\langle1,b_s\rangle-B)\in A$.}
$$
The latter conditions are schematized as
$$
1\;\;\cdots\;\;(\overline{a}-1)\;\underbrace{\overline{a}\raisebox{-1mm}{}}_{\text{in $A$}}\;\underbrace{(\overline{a}+1)\;\;\cdots\;\;(b_s-1)\;\;b_s}_{\text{$b_s-\overline{a}$ elements in $B$}}\;\;\underbrace{(b_s+1)\;\;\cdots\;\;(\mu-1)}_{\text{$\mu-1-b_s$ elements in $A$}},
$$
which makes it clear that the double summation in~(\ref{fea}) counts the cells in $\mathcal{C}$ (the index $b$ in the outer summation stands for $b_s$, and the index $a$ in the inner summation stands for $\overline{a}$).
\end{proof}

\begin{proposition}\label{Betticompletegraphs}
Assume $\mu<2\kappa-1$ and $\mu-\kappa-1=1$. Then $2\binom{\mu-1}{\mu-\kappa}$ and~(\ref{fea}) give, respectively, the first and second betti numbers of the 0-connected 2-dimensional space $X_{\mu,\kappa}$.
\end{proposition}
\begin{proof}
Under the current hypothesis, $X_{\mu,\kappa}$ is the discrete configuration space of two ordered points in the complete graph $K_\mu$ on $\mu$ vertices, with $\mu\geq5$. The result then follows, after a straightforward calculation, from Proposition~21 and equations~(1) and~(2) in~\cite{MR2745669}.
\end{proof}

\section{Gradient flow}\label{homtopviamorse}
We now prove the second assertion in item~\emph{(\ref{it4})} of Theorem~\ref{maintheoremlarge} so, throughout this section, we assume $1<d=m-3$, i.e., $\kappa=3$ and $\mu\geq6$. We also use the more legible notation $a_1,\etc,a_r\y b_1,\etc,b_s$ as an alternative for the cell $(a_1\cdots a_r,b_1\cdots b_s)$ in~(\ref{acotado}). Thus, the top ($\mu-3$)-dimensional critical cell~(\ref{topcell}) in $X_{\mu,\kappa}$ is
\begin{equation}\label{arriba}
\mu-2,\mu-1\y1,\etc,\mu-3,
\end{equation}
while an easy check (using the descriptions in Corollary~\ref{celdascriticas}) gives that the ($\mu-4$)-dimensional critical cells are
\begin{itemize}
\item[\emph{(iii)}] $1,\etc,\widehat{i},\etc,\widehat{b-1},\widehat{b},\etc,\mu\y b$, for $1\leq i\leq b-2$ and $b\leq\mu$,
\item[\emph{(iv)}] $\mu\y1,\etc,\widehat{i},\etc\mu-2$, for $1\leq i\leq \mu-2$,
\item[\emph{(vii)}] $a\y1,\etc,\widehat{i},\etc,\widehat{a},\etc,\mu-1$, for $1\leq i<a\leq\mu-2$,
\end{itemize}
where a hat is used to indicate removal, and item enumeration follows the one used in Corollary~\ref{celdascriticas}.

\medskip
The homotopy triviality of the relevant attaching map $S^{\mu-4}\to\bigvee_{}S^{\mu-4}$ can be verified homologically (as $\mu\geq6$). By~(\ref{morseboundary}), we will be done once the following fact is proved:

\begin{proposition}\label{cruxofthematter}
Fix a $(\mu-4)$-dimensional critical cell $\alpha$. Taking into account multiplicity and incidence numbers as in~(\ref{incidencealamorse}), the total number of gradient paths from faces of~(\ref{arriba}) to $\alpha$ is zero.
\end{proposition}

For the proof, we think of gradient paths in Proposition~\ref{cruxofthematter} as \emph{lower} gradient paths from~(\ref{arriba}) to 1-cofaces $\beta$ of $\alpha$, i.e., with $\beta\searrow\alpha$, and make a \emph{backwards} description of the latter paths. For starters, no such $\beta$ is critical, in view of Theorem~\ref{maintheoremmaximal}. Furthermore, the lower paths we care about involve a collapsible $\beta$, say with $\alpha_1\nearrow\beta$. So, we repeat the description strategy now for any possible non-redundant 1-coface $\beta_1$ of $\alpha_1$, 
$$
\beta_1\searrow\alpha_1\nearrow\beta\searrow\alpha.
$$
This creates a tree $T_\alpha$ describing all possible paths in Proposition~\ref{cruxofthematter}, with each branch stopping as soon as we get a $\beta_k\searrow\alpha_k$ with $\beta_k$ the critical face~(\ref{arriba}). Much of the hard work in the analysis of $T_\alpha$ is greatly simplified by the next two auxiliary results.

\begin{lemma}\label{down}
There are no lower gradient paths
\begin{equation}\label{pth}
\alpha_0\searrow\beta_1\nearrow\alpha_1\searrow\cdots\searrow \beta_k\nearrow\alpha_k
\end{equation}
starting at the critical $\alpha_0=\mu-2,\mu-1\y1,\etc,\mu-3$ and ending at
\begin{enumerate}[(a)]
\item $\alpha_k=\mu-j+1,\mu-j+2,\etc,\mu\y1,\etc,\widehat{i},\etc,\mu-j$, \,if \,$3\leq j\leq\mu-3$ and $1\leq i\leq\mu-j$. 
\item $\alpha_k=\ell,\mu-j+1,\mu-j+2,\etc,\mu\y1,\etc,\widehat{i},\etc,\widehat{\ell},\etc,\mu-j$, \,if \,$2\leq j\leq\mu-4$ and $1\leq i<\ell\leq\mu-j$. 
\item $\alpha_k=i,\mu-j+1,\mu-j+2,\etc,\mu\y1,\etc,\widehat{i},\etc,\widehat{\ell},\etc,\mu-j$, \,if \,$2\leq j\leq\mu-4$ and $1\leq i<\ell\leq\mu-j$. 
\end{enumerate}
\end{lemma}
\begin{proof}
\emph{(a)} We proceed by induction on the difference $\mu-3-j\geq0$, noticing that we can always assume $i<\mu-j$, for otherwise $\alpha_k$ is redundant in view of Corollary~\ref{celdasredundantes}(\ref{teorema1}). The induction starts with the case $j=\mu-3$, for which a potential path~(\ref{pth}) is ruled out as it would be forced to end in the form
\begin{equation}\label{empezando}
4,\etc,\mu\y1,2\searrow4,\etc,\mu\y1,\etc,\widehat{i},\etc,\widehat{3}\stackrel{\ref{teorema1}}\nearrow4,\etc,\mu\y1,\etc,\widehat{i},\etc,3=\alpha_k,
\end{equation}
with $4,\etc,\mu\y1,2$ redundant in view of Corollary~\ref{celdasredundantes}(\ref{teorema1}). Here and below, the number on top of an up-going directed edge is the label of the fact granting the edge. Furthermore, the reader should keep in mind that, as explained in the construction of the tree $T_\alpha$, this type of gradient paths \emph{must} be verified through a right-to-left reading, as we are restricting attention to paths landing on $\alpha_k$. Note that the down-going edge in~(\ref{empezando}) is forced because $\{4,\etc,\mu\}$ has maximal allowed cardinality; such a situation will repeat several times in the considerations below. Now, having clarified the previous points, note that, for $j<\mu-3$, a potential path~(\ref{pth}) is forced to end in either of the three forms
\begin{align*}
\left.\begin{array}{r}
\beta_1:=\mu-j+1,\etc,\mu\y1,\etc,\mu-j-1\\
\beta_2:=\mu-j,\etc,\mu\y1,\etc,\widehat{i},\etc,\mu-j-1\\
\beta_3:=i,\mu-j+1,\etc,\mu\y1,\etc,\widehat{i},\etc,\mu-j-1
\end{array}\right\}&\searrow\mu-j+1,\etc,\mu\y1,\etc,\widehat{i},\etc,\mu-j-1\\&\stackrel{\ref{teorema1}}\nearrow\mu-j+1,\etc,\mu\y1,\etc,\widehat{i},\etc,\mu-j=\alpha_k.
\end{align*}
But the first and third options are ruled out since $\beta_1$ and $\beta_3$ are redundant in view of Corollary~\ref{celdasredundantes}(\ref{teorema1}), while the second option is ruled out by induction.

\smallskip
The arguments for~\emph{(b)} and~\emph{(c)} have only minor variants. We give full details for completeness. 

\smallskip
\emph{(b)} We proceed by induction on the difference $\mu-4-j\geq0$, noticing that we can always assume $\ell<\mu-j$, for otherwise the result reduces to that in~\emph{(a)}. The induction starts with the case $j=\mu-4$, for which a potential path~(\ref{pth}) is ruled out as it would be forced to end in the form
$$
\ell,5,\etc,\mu\y1,\etc,\widehat{\ell},\etc,3
\searrow
\ell,5,\etc,\mu\y1,\etc,\widehat{i},\etc,\widehat{\ell},\etc,3
\stackrel{\ref{teorema1}}\nearrow
\ell,5,\etc,\mu\y1,\etc,\widehat{i},\etc,\widehat{\ell},\etc,4=\alpha_k,
$$
with $\ell,5,\etc,\mu\y1,\etc,\widehat{\ell},\etc,3$ redundant in view of Corollary~\ref{celdasredundantes}(\ref{teorema1}). Then, assuming $j<\mu-4$, a potential path~(\ref{pth}) is forced to end in either of the three forms
\begin{align*}
&\left.\begin{array}{r}
\beta_1:=i,\ell,\mu-j+1,\etc,\mu\y1,\etc,\widehat{i},\etc,\widehat{\ell},\etc,\mu-j-1\\
\beta_2:=\ell,\mu-j,\etc,\mu\y1,\etc,\widehat{i},\etc,\widehat{\ell},\etc\mu-j-1\\
\beta_3:=\ell,\mu-j+1,\etc,\mu\y1,\etc,\widehat{\ell},\etc,\mu-j-1
\end{array}\right\}\searrow\\
&\rule{80mm}{0mm}\searrow\ell,\mu-j+1,\etc,\mu\y1,\etc,\widehat{i},\etc,\widehat{\ell},\etc,\mu-j-1\\
&\rule{80mm}{0mm}\stackrel{\ref{teorema1}}\nearrow\ell,\mu-j+1,\etc,\mu\y1,\etc,\widehat{i},\etc,\widehat{\ell},\etc,\mu-j=\alpha_k.
\end{align*}
But the first and third options are ruled out since $\beta_1$ and $\beta_3$ are redundant in view of Corollary~\ref{celdasredundantes}(\ref{teorema1}), while the second option is ruled out by induction.

\smallskip
\emph{(c)} We proceed by induction on the difference $\mu-4-j\geq0$, noticing that we can always assume $\ell<\mu-j$, for otherwise $\alpha_k$ is redundant in view of Corollary~\ref{celdasredundantes}(\ref{teorema1}). The induction starts with the case $j=\mu-4$, for which a potential path~(\ref{pth}) is ruled out as it would be forced to end in the form
$$
i,5,\etc,\mu\y1,\etc,\widehat{i},\etc,3
\searrow
i,5,\etc,\mu\y1,\etc,\widehat{i},\etc,\widehat{\ell},\etc,3
\stackrel{\ref{teorema1}}\nearrow
i,5,\etc,\mu\y1,\etc,\widehat{i},\etc,\widehat{\ell},\etc,4=\alpha_k,
$$
with $i,5,\etc,\mu\y1,\etc,\widehat{i},\etc,3$ redundant in view of Corollary~\ref{celdasredundantes}(\ref{teorema1}). Then, assuming $j<\mu-4$, a potential path~(\ref{pth}) is forced to end in either of the three forms
\begin{align*}
&\left.\begin{array}{r}
\beta_1:=i,\ell,\mu-j+1,\etc,\mu\y1,\etc,\widehat{i},\etc,\widehat{\ell},\etc,\mu-j-1\\
\beta_2:=i,\mu-j,\etc,\mu\y1,\etc,\widehat{i},\etc,\widehat{\ell},\etc\mu-j-1\\
\beta_3:=i,\mu-j+1,\etc,\mu\y1,\etc,\widehat{i},\etc,\mu-j-1
\end{array}\right\}\searrow\\
&\rule{80mm}{0mm}\searrow i,\mu-j+1,\etc,\mu\y1,\etc,\widehat{i},\etc,\widehat{\ell},\etc,\mu-j-1\\
&\rule{80mm}{0mm}\stackrel{\ref{teorema1}}\nearrow i,\mu-j+1,\etc,\mu\y1,\etc,\widehat{i},\etc,\widehat{\ell},\etc,\mu-j=\alpha_k.
\end{align*}
But the first and third options are ruled out since $\beta_1$ and $\beta_3$ are redundant in view of Corollary~\ref{celdasredundantes}(\ref{teorema1}), while the second option is ruled out by induction.
\end{proof}

The case $j=2$ in item~\emph{(a)} of Lemma~\ref{down} is special:

\begin{lemma}\label{unica}
For $1\leq\ell\leq\mu-2$, there is a single lower gradient path~(\ref{pth}) starting at the critical face $\alpha_0=\mu-2,\mu-1\y1,\etc,\mu-3$ and ending at $\alpha_k=\mu-1,\mu\y1,\etc,\widehat{\ell},\etc,\mu-2$.
\end{lemma}
\begin{proof}
The result is easy for $\ell=\mu-2$, as the only such path~(\ref{pth}) is
\begin{equation}\label{solita}
\mu-2,\mu-1\y1,\etc,\mu-3
\searrow
\mu-1\y1,\etc,\mu-3
\stackrel{\ref{teorema2}(2.a)}\nearrow
\mu-1,\mu\y1,\etc,\mu-3=\alpha_k.
\end{equation}
For $\ell\leq\mu-3$, a potential path~(\ref{pth}) is forced to end in the form
\begin{align}
\left.\begin{array}{r}
\beta_1:=\ell,\mu-1,\mu\y1,\etc,\widehat{\ell},\etc,\mu-3\\
\beta_2:=\mu-2,\mu-1,\mu\y1,\etc,\widehat{\ell},\etc,\mu-3\\
\beta_3:=\mu-1,\mu\y1,\etc,\mu-3
\end{array}\right\}
&\searrow
\mu-1,\mu\y1,\etc,\widehat{\ell},\etc,\mu-3\nonumber\\
&\stackrel{\ref{teorema1}}\nearrow
\mu-1,\mu\y1,\etc,\widehat{\ell},\etc,\mu-2=\alpha_k.\label{solitabis}
\end{align}
But the first option is ruled out since $\beta_1$ is redundant in view of Corollary~\ref{celdasredundantes}(\ref{teorema1}), while the second option is ruled out in view of item~\emph{(a)} in Lemma~\ref{down}. The result then follows by noticing that the third option is forced to end up as in~(\ref{solita}).
\end{proof}

\begin{proof}[Proof of Proposition~\ref{cruxofthematter}]
Start by considering the case of a $(\mu-4)$-dimensional type-\emph{(iii)} critical face $\alpha=1,\etc,\widehat{i},\etc,\widehat{b-1},\widehat{b},\etc,\mu\y b$, with $1\leq i\leq b-2$ and $b\leq\mu$. Note that $\alpha$ is face of only two faces of dimension $\mu-3$, namely, $1,\etc,\widehat{i},\etc,\widehat{b-1},\widehat{b},\etc,\mu\y b-1,b$ and $1,\etc,\widehat{i},\etc,\widehat{b-1},\widehat{b},\etc,\mu\y i,b$. Therefore, the analysis of the tree $T_\alpha$ starts with what are the only two potential branches
\begin{align*}
1,\etc,\widehat{i},\etc,\widehat{b-1},\widehat{b},\etc,\mu\y i,b-1
&\searrow
1,\etc,\widehat{i},\etc,\widehat{b-1},\widehat{b},\etc,\mu\y b-1 \\
&\stackrel{\ref{teorema1}}\nearrow
1,\etc,\widehat{i},\etc,\widehat{b-1},\widehat{b},\etc,\mu\y b-1,b \\
&\searrow
1,\etc,\widehat{i},\etc,\widehat{b-1},\widehat{b},\etc,\mu\y b=\alpha
\end{align*}
and
\begin{align*}
1,\etc,\widehat{i},\etc,\widehat{b-1},\widehat{b},\etc,\mu\y i,b-1
&\searrow
1,\etc,\widehat{i},\etc,\widehat{b-1},\widehat{b},\etc,\mu\y i \\
&\stackrel{\ref{teorema1}}\nearrow
1,\etc,\widehat{i},\etc,\widehat{b-1},\widehat{b},\etc,\mu\y i,b \\
&\searrow
1,\etc,\widehat{i},\etc,\widehat{b-1},\widehat{b},\etc,\mu\y b=\alpha.
\end{align*}
But $1,\etc,\widehat{i},\etc,\widehat{b-1},\widehat{b},\etc,\mu\y i,b-1$ is redundant by Corollary~\ref{celdasredundantes}(\ref{teorema1}), so $T_\alpha$ is in fact empty.

\smallskip
Consider next the case of a $(\mu-4)$-dimensional type-\emph{(iv)} critical face $\alpha=\mu\y1,\etc,\widehat{i},\etc,\mu-2$, with $1\leq i\leq \mu-2$. As above, the analysis of the tree $T_\alpha$ starts with precisely two potential branches. For $i=\mu-2$, each of those two branches is linear and leads to an actual path in $T_\alpha$. Indeed, $T_\alpha$ reduces to the two branches
$$
\mu-2,\mu-1\y1,\etc,\mu-3
\searrow
\mu-2\y1,\etc,\mu-3
\stackrel{\ref{teorema2}(2.a)}\nearrow
\mu-2,\mu\y1,\etc,\mu-3
\searrow
\mu\y1,\etc,\mu-3=\alpha
$$
and
$$
\mu-2,\mu-1\y1,\etc,\mu-3
\searrow
\mu-1\y1,\etc,\mu-3
\stackrel{\ref{teorema2}(2.a)}\nearrow
\mu-1,\mu\y1,\etc,\mu-3
\searrow
\mu\y1,\etc,\mu-3=\alpha,
$$
each starting at the required~(\ref{arriba}). The result then follows from~(\ref{sgns}) and a straightforward analysis of incidence numbers in the two paths above: the first branch accounts in~(\ref{incidencealamorse}) with a negative sign, while the second branch does it with a positive sign. On the other hand, for $1\leq i\leq \mu-3$, the analysis of the tree $T_\alpha$ starts with its only two potential branches
\begin{align*}
\left.\begin{array}{r}
\beta_1:=\mu-1,\mu\y1,\etc,\mu-3\\
\beta_2:=\mu-2,\mu-1,\mu\y1,\etc,\widehat{i},\etc,\mu-3\\
\beta_3:=i,\mu-1,\mu\y1,\etc,\widehat{i},\etc,\mu-3
\end{array}\right\}
&\searrow
\mu-1,\mu\y1,\etc,\widehat{i},\etc,\mu-3\\
&\stackrel{\ref{teorema1}}
\nearrow
\mu-1,\mu\y1,\etc,\widehat{i},\etc,\mu-2\\&
\searrow
\mu\y1,\etc,\widehat{i},\etc,\mu-2=\alpha
\end{align*}
and
\begin{align*}
\left.\begin{array}{r}
\beta_1=\mu-1,\mu\y1,\etc,\mu-3\\
\beta_2=\mu-2,\mu-1,\mu\y1,\etc,\widehat{i},\etc,\mu-3\\
\beta_3=i,\mu-1,\mu\y1,\etc,\widehat{i},\etc,\mu-3
\end{array}\right\}
&\searrow
\mu-1,\mu\y1,\etc,\widehat{i},\etc,\mu-3\\
&\stackrel{\ref{teorema1}}
\nearrow
\mu-1,\mu\y1,\etc,\widehat{i},\etc,\mu-2\\&
\searrow
\mu-1\y1,\etc,\widehat{i},\etc,\mu-2\\&
\stackrel{\ref{teorema2}(2.b)}
\nearrow
i,\mu-1\y1,\etc,\widehat{i},\etc,\mu-2\\&
\searrow
i\y1,\etc,\widehat{i},\etc,\mu-2\\&
\stackrel{\ref{teorema2}(2.a)}
\nearrow
i,\mu\y1,\etc,\widehat{i},\etc,\mu-2\\&
\searrow
\mu\y1,\etc,\widehat{i},\etc,\mu-2=\alpha.
\end{align*}
But none of the two options involving $\beta_2$  leads to an actual path in $T_\alpha$, in view of item~\emph{(a)} in Lemma~\ref{down}. Likewise, none of the two options involving $\beta_3$ leads to an actual path in $T_\alpha$, for $\beta_3$ is redundant by Corollary~\ref{celdasredundantes}(\ref{teorema1}). As above, the result then follows from a straightforward analysis of incidence numbers: the two options involving~$\beta_1$ allow us to pair the actual paths in $T_\alpha$ in such a way that the two elements in each pair account in~(\ref{incidencealamorse}) with different signs. (It is easy to see that there is only one pair, but such a fact is immaterial for our argument.)

\smallskip
To complete the proof, we next deal with the slightly more involved case of a $(\mu-4)$-dimensional type-\emph{(vii)} critical face $\alpha=a\y1,\etc,\widehat{i},\etc,\widehat{a},\etc,\mu-1$, with $1\leq i< a\leq \mu-2$. As above, the analysis of the corresponding tree $T_\alpha$ starts with its only two potential branches
\begin{align*}
&\left.\begin{array}{r}
\beta_1:=i,a,\mu\y1,\etc,\widehat{i},\etc,\widehat{a},\etc,\mu-2\\
\beta_2:=a,\mu-1,\mu\y1,\etc,\widehat{i},\etc,\widehat{a},\etc,\mu-2\\
\beta_3:=a,\mu\y1,\etc,\widehat{a},\etc,\mu-2
\end{array}\right\}
\searrow
a,\mu\y1,\etc,\widehat{i},\etc,\widehat{a},\etc,\mu-2\\
&\hspace{5.3cm}\stackrel{\ref{teorema1}}
\nearrow
a,\mu\y1,\etc,\widehat{i},\etc,\widehat{a},\etc,\mu-1
\searrow
a\y1,\etc,\widehat{i},\etc,\widehat{a},\etc,\mu-1=\alpha
\end{align*}
and
\begin{align*}
&\left.\begin{array}{r}
\beta_1=i,a,\mu\y1,\etc,\widehat{i},\etc,\widehat{a},\etc,\mu-2\\
\beta'_2:=i,\mu-1,\mu\y1,\etc,\widehat{i},\etc,\widehat{a},\etc,\mu-2\\
\beta'_3:=i,\mu\y1,\etc,\widehat{i},\etc,\mu-2
\end{array}\right\}
\searrow
i,\mu\y1,\etc,\widehat{i},\etc,\widehat{a},\etc,\mu-2
\\&\hspace{5.1cm}
\stackrel{\ref{teorema1}}
\nearrow
i,\mu\y1,\etc,\widehat{i},\etc,\widehat{a},\etc,\mu-1
\searrow
i\y1,\etc,\widehat{i},\etc,\widehat{a},\etc,\mu-1\\&\hspace{4.8cm}
\stackrel{\ref{teorema2}(2.b)}
\nearrow
i,a\y1,\etc,\widehat{i},\etc,\widehat{a},\etc,\mu-1
\searrow
a\y1,\etc,\widehat{i},\etc,\widehat{a},\etc,\mu-1=\alpha.
\end{align*}
None of the two options involving $\beta_1$ leads to an actual path in $T_\alpha$, for $\beta_1$ is redundant by Corollary~\ref{celdasredundantes}(\ref{teorema1}). Likewise, none of the two options involving $\beta_2$ or $\beta'_2$ leads to an actual path in $T_\alpha$, in view of items~\emph{(b)} and~\emph{(c)} in Lemma~\ref{down}. This time, however, we can not use the pairing trick in the previous case, as $\beta_3\neq\beta'_3$. Instead, we continue ``one step further'' the analysis of $T_\alpha$ which, all in all, starts in fact with its only two potential (but now linear) branches 
\begin{align*}
&
\mu-1,\mu\y1,\etc,\widehat{a},\etc,\mu-2
\searrow
\mu-1\y1,\etc,\widehat{a},\etc,\mu-2\\&
\stackrel{\ref{teorema2}(2.b)}\nearrow
a,\mu-1\y1,\etc,\widehat{a},\etc,\mu-2
\searrow
a\y1,\etc,\widehat{a},\etc,\mu-2
\stackrel{\ref{teorema2}(2.a)}\nearrow
a,\mu\y1,\etc,\widehat{a},\etc,\mu-2
\\&\searrow
a,\mu\y1,\etc,\widehat{i},\etc,\widehat{a},\etc,\mu-2
\stackrel{\ref{teorema1}}
\nearrow
a,\mu\y1,\etc,\widehat{i},\etc,\widehat{a},\etc,\mu-1
\searrow
a\y1,\etc,\widehat{i},\etc,\widehat{a},\etc,\mu-1=\alpha
\end{align*}
and
\begin{align*}
&\mu-1,\mu\y1,\etc,\widehat{i},\etc,\mu-2
\searrow
\mu-1\y1,\etc,\widehat{i},\etc,\mu-2
\stackrel{\ref{teorema2}(2.b)}\nearrow
i,\mu-1\y1,\etc,\widehat{i},\etc,\mu-2\\&
\searrow
i\y1,\etc,\widehat{i},\etc,\mu-2
\stackrel{\ref{teorema2}(2.a)}\nearrow
i,\mu\y1,\etc,\widehat{i},\etc,\mu-2
\searrow
i,\mu\y1,\etc,\widehat{i},\etc,\widehat{a},\etc,\mu-2\\&
\stackrel{\ref{teorema1}}\nearrow
i,\mu\y1,\etc,\widehat{i},\etc,\widehat{a},\etc,\mu-1
\searrow
i\y1,\etc,\widehat{i},\etc,\widehat{a},\etc,\mu-1
\stackrel{\ref{teorema2}(2.b)}
\nearrow
i,a\y1,\etc,\widehat{i},\etc,\widehat{a},\etc,\mu-1\\&
\searrow
a\y1,\etc,\widehat{i},\etc,\widehat{a},\etc,\mu-1=\alpha.
\end{align*}
But Lemma~\ref{unica} implies that each of the two branches above is forced to encode a single and fully explicit path in $T_\alpha$. The conclusion then follows from a straightforward analysis of the resulting incidence numbers. Indeed, as the reader will readily check, the first branch encodes a path which accounts in~(\ref{incidencealamorse}) with a sign $(-1)^{i+a}$, while the second branch does it with sign $(-1)^{i+a+1}$.
\end{proof}


\bigskip

{\sc \ 

Departamento de Matem\'aticas

Centro de Investigaci\'on y de Estudios Avanzados del I.P.N.

Av.~Instituto Polit\'ecnico Nacional n\'umero~2508, San Pedro Zacatenco

M\'exico City 07000, M\'exico.}

\tt jesus@math.cinvestav.mx

\tt jgonzalez@math.cinvestav.mx

\end{document}